\documentclass[a4paper,11pt,reqno]{amsart}

\usepackage[utf8]{inputenc}
\usepackage[T1]{fontenc}
\usepackage{hyperref}
\usepackage{xcolor}
\hypersetup{
    colorlinks,
    linkcolor={red!50!black},
    citecolor={blue!50!black},
    urlcolor={blue!80!black}
}

\usepackage{microtype}
\usepackage{MnSymbol}
\usepackage{textcomp}

\usepackage[mathcal]{euscript}
\let\mathcaltmp\mathcal
\usepackage{mathrsfs}
\let\mathcal\mathscr
\let\mathscr\mathcaltmp

\makeatletter
\def\thm@space@setup{\thm@preskip=7pt
\thm@postskip=7pt}
\makeatother
\newtheoremstyle{plain}
  {}
  {}
  {\slshape}
  {}
  {\bfseries}
  {.}
  { }
  {}

\newtheoremstyle{definition}
  {}
  {}
  {}
  {}
  {\bfseries}
  {.}
  { }
  {}

\makeatletter
\renewenvironment{proof}[1][\proofname]{\par
  \pushQED{\qed}
  \normalfont \topsep0\p@\relax
  \trivlist
  \item[\hskip\labelsep\itshape
  #1\@addpunct{.}]\ignorespaces
}{
  \popQED\endtrivlist\@endpefalse
}
\makeatother

\makeatletter
\newcommand{\eqnum}{\refstepcounter{equation}\textup{\tagform@{\theequation}}}
\makeatother

\makeatletter \@addtoreset{equation}{section} \makeatother
\renewcommand{\theequation}{\thesection.\arabic{equation}}

\newtheorem{thm}[equation]{Theorem}
\newtheorem{thmX}{Theorem}

\newtheorem{lem}[equation]{Lemma}
\newtheorem{cor}[equation]{Corollary}
\newtheorem{prop}[equation]{Proposition}
\newtheorem*{defthm*}{Definition/Theorem}
\newtheorem{conj}[equation]{Conjecture}
\newtheorem{quest}[equation]{Question}
\theoremstyle{definition}
\newtheorem{defn}[equation]{Definition}
\newtheorem{rem}[equation]{Remark}
\newtheorem{exam}[equation]{Example}

\newtheorem*{exam*}{Example}

\newcommand\arXiv[1]{\href{http://arxiv.org/abs/#1}{arXiv:#1}}

\usepackage{xparse}

\usepackage{xspace}

\newcommand{\changelocaltocdepth}[1]{
  \addtocontents{toc}{\protect\setcounter{tocdepth}{#1}}
  \setcounter{tocdepth}{#1}}

\usepackage{xpatch}
\makeatletter   
\xpatchcmd{\@tocline}
  {\hfil\hbox to\@pnumwidth{\@tocpagenum{#7}}\par}
  {\ifnum#1<0\hfill\else\dotfill\fi\hbox to\@pnumwidth{\@tocpagenum{#7}}\par}
  {}{}
\def\l@subsection{\@tocline{2}{0pt}{4pc}{6pc}{}}
\def\l@subsubsection{\@tocline{3}{0pt}{8pc}{8pc}{}}
\makeatother

\newcommand{\nc}{\newcommand}
\nc{\renc}{\renewcommand}
\nc{\ssec}{\subsection}
\nc{\sssec}{\subsubsection}
\nc{\on}{\operatorname}
\nc{\term}[1]{#1\xspace}

\usepackage{tikz}
\usetikzlibrary{matrix}
\usepackage{tikz-cd}

\tikzset{
  commutative diagrams/.cd,
  arrow style=tikz,
  diagrams={>=latex}}
\makeatletter
\tikzset{
  column sep/.code=\def\pgfmatrixcolumnsep{\pgf@matrix@xscale*(#1)},
  row sep/.code   =\def\pgfmatrixrowsep{\pgf@matrix@yscale*(#1)},
  matrix xscale/.code=
    \pgfmathsetmacro\pgf@matrix@xscale{\pgf@matrix@xscale*(#1)},
  matrix yscale/.code=
    \pgfmathsetmacro\pgf@matrix@yscale{\pgf@matrix@yscale*(#1)},
  matrix scale/.style={/tikz/matrix xscale={#1},/tikz/matrix yscale={#1}}}
\def\pgf@matrix@xscale{1}
\def\pgf@matrix@yscale{1}
\makeatother

\usepackage[inline]{enumitem}
\setlist[enumerate,1]{label={(\alph*)},itemsep=\parskip}

\newlist{thmlist}{enumerate}{1}
\setlist[thmlist,1]{
  label={\em(\roman*)}, ref={(\roman*)},
  itemsep=0.5em,
  topsep=0em,
  leftmargin=*,
  align=left,widest=vi)}

\newlist{thmlistbis}{enumerate}{1}
\setlist[thmlistbis,1]{
  label={\em(\roman*~\textit{bis})},
  ref={(\roman*}~\textit{bis}\upshape{)},
  itemsep=0.5em,
  topsep=-0.7em,
  leftmargin=0pt, align=right, widest=vi)}

\newlist{defnlist}{enumerate}{2}
\setlist[defnlist,1]{
  label={(\roman*)}, ref={(\roman*)},
  itemsep=0.5em,
  topsep=0em,
  leftmargin=*,
  align=left, widest=vi)}

\setlist[defnlist,2]{
  label={(\alph*)}, ref={(\alph*)},
  itemsep=0.75em,
  labelsep=0em,labelindent=0em,leftmargin=*,align=left,widest=vi),
  topsep=0.75em}

\newlist{defnlistbis}{enumerate}{1}
\setlist[defnlistbis,1]{
  label={(\roman*~\textit{bis})},
  ref={(\roman*}~\textit{bis}\upshape{)},
  itemsep=0.5em,
  topsep=0em,
  leftmargin=*,
  align=left, widest=vi)}

\newlist{inlinelist}{enumerate*}{1}
\setlist[inlinelist,1]{label={(\alph*)}}

\newlist{inlinedefnlist}{enumerate*}{1}
\definecolor{green}{HTML}{38550C}
\setlist[inlinedefnlist,1]{label={\color{green}(\roman*)}}

\newlist{inlinethmlist}{enumerate*}{1}
\definecolor{green}{HTML}{38550C}
\setlist[inlinethmlist,1]{label={\color{green}(\roman*)}}

\nc{\cA}{\ensuremath{\mathcal{A}}\xspace}
\nc{\cB}{\ensuremath{\mathcal{B}}\xspace}
\nc{\cC}{\ensuremath{\mathcal{C}}\xspace}
\nc{\cD}{\ensuremath{\mathcal{D}}\xspace}
\nc{\cE}{\ensuremath{\mathcal{E}}\xspace}
\nc{\cF}{\ensuremath{\mathcal{F}}\xspace}
\nc{\cG}{\ensuremath{\mathcal{G}}\xspace}
\nc{\cH}{\ensuremath{\mathcal{H}}\xspace}
\nc{\cI}{\ensuremath{\mathcal{I}}\xspace}
\nc{\cJ}{\ensuremath{\mathcal{J}}\xspace}
\nc{\cK}{\ensuremath{\mathcal{K}}\xspace}
\nc{\cL}{\ensuremath{\mathcal{L}}\xspace}
\nc{\cM}{\ensuremath{\mathcal{M}}\xspace}
\nc{\cN}{\ensuremath{\mathcal{N}}\xspace}
\nc{\cO}{\ensuremath{\mathcal{O}}\xspace}
\nc{\cP}{\ensuremath{\mathcal{P}}\xspace}
\nc{\cQ}{\ensuremath{\mathcal{Q}}\xspace}
\nc{\cR}{\ensuremath{\mathcal{R}}\xspace}
\nc{\cS}{\ensuremath{\mathcal{S}}\xspace}
\nc{\cT}{\ensuremath{\mathcal{T}}\xspace}
\nc{\cU}{\ensuremath{\mathcal{U}}\xspace}
\nc{\cV}{\ensuremath{\mathcal{V}}\xspace}
\nc{\cW}{\ensuremath{\mathcal{W}}\xspace}
\nc{\cX}{\ensuremath{\mathcal{X}}\xspace}
\nc{\cY}{\ensuremath{\mathcal{Y}}\xspace}
\nc{\cZ}{\ensuremath{\mathcal{Z}}\xspace}

\nc{\sA}{\ensuremath{\mathscr{A}}\xspace}
\nc{\sB}{\ensuremath{\mathscr{B}}\xspace}
\nc{\sC}{\ensuremath{\mathscr{C}}\xspace}
\nc{\sD}{\ensuremath{\mathscr{D}}\xspace}
\nc{\sE}{\ensuremath{\mathscr{E}}\xspace}
\nc{\sF}{\ensuremath{\mathscr{F}}\xspace}
\nc{\sG}{\ensuremath{\mathscr{G}}\xspace}
\nc{\sH}{\ensuremath{\mathscr{H}}\xspace}
\nc{\sI}{\ensuremath{\mathscr{I}}\xspace}
\nc{\sJ}{\ensuremath{\mathscr{J}}\xspace}
\nc{\sK}{\ensuremath{\mathscr{K}}\xspace}
\nc{\sL}{\ensuremath{\mathscr{L}}\xspace}
\nc{\sM}{\ensuremath{\mathscr{M}}\xspace}
\nc{\sN}{\ensuremath{\mathscr{N}}\xspace}
\nc{\sO}{\ensuremath{\mathscr{O}}\xspace}
\nc{\sP}{\ensuremath{\mathscr{P}}\xspace}
\nc{\sQ}{\ensuremath{\mathscr{Q}}\xspace}
\nc{\sR}{\ensuremath{\mathscr{R}}\xspace}
\nc{\sS}{\ensuremath{\mathscr{S}}\xspace}
\nc{\sT}{\ensuremath{\mathscr{T}}\xspace}
\nc{\sU}{\ensuremath{\mathscr{U}}\xspace}
\nc{\sV}{\ensuremath{\mathscr{V}}\xspace}
\nc{\sW}{\ensuremath{\mathscr{W}}\xspace}
\nc{\sX}{\ensuremath{\mathscr{X}}\xspace}
\nc{\sY}{\ensuremath{\mathscr{Y}}\xspace}
\nc{\sZ}{\ensuremath{\mathscr{Z}}\xspace}

\nc{\bA}{\ensuremath{\mathbf{A}}\xspace}
\nc{\bB}{\ensuremath{\mathbf{B}}\xspace}
\nc{\bC}{\ensuremath{\mathbf{C}}\xspace}
\nc{\bD}{\ensuremath{\mathbf{D}}\xspace}
\nc{\bE}{\ensuremath{\mathbf{E}}\xspace}
\nc{\bF}{\ensuremath{\mathbf{F}}\xspace}
\nc{\bG}{\ensuremath{\mathbf{G}}\xspace}
\nc{\bH}{\ensuremath{\mathbf{H}}\xspace}
\nc{\bI}{\ensuremath{\mathbf{I}}\xspace}
\nc{\bJ}{\ensuremath{\mathbf{J}}\xspace}
\nc{\bK}{\ensuremath{\mathbf{K}}\xspace}
\nc{\bL}{\ensuremath{\mathbf{L}}\xspace}
\nc{\bM}{\ensuremath{\mathbf{M}}\xspace}
\nc{\bN}{\ensuremath{\mathbf{N}}\xspace}
\nc{\bO}{\ensuremath{\mathbf{O}}\xspace}
\nc{\bP}{\ensuremath{\mathbf{P}}\xspace}
\nc{\bQ}{\ensuremath{\mathbf{Q}}\xspace}
\nc{\bR}{\ensuremath{\mathbf{R}}\xspace}
\nc{\bS}{\ensuremath{\mathbf{S}}\xspace}
\nc{\bT}{\ensuremath{\mathbf{T}}\xspace}
\nc{\bU}{\ensuremath{\mathbf{U}}\xspace}
\nc{\bV}{\ensuremath{\mathbf{V}}\xspace}
\nc{\bW}{\ensuremath{\mathbf{W}}\xspace}
\nc{\bX}{\ensuremath{\mathbf{X}}\xspace}
\nc{\bY}{\ensuremath{\mathbf{Y}}\xspace}
\nc{\bZ}{\ensuremath{\mathbf{Z}}\xspace}

\nc{\bbA}{\ensuremath{\mathbb{A}}\xspace}
\nc{\bbB}{\ensuremath{\mathbb{B}}\xspace}
\nc{\bbC}{\ensuremath{\mathbb{C}}\xspace}
\nc{\bbD}{\ensuremath{\mathbb{D}}\xspace}
\nc{\bbE}{\ensuremath{\mathbb{E}}\xspace}
\nc{\bbF}{\ensuremath{\mathbb{F}}\xspace}
\nc{\bbG}{\ensuremath{\mathbb{G}}\xspace}
\nc{\bbH}{\ensuremath{\mathbb{H}}\xspace}
\nc{\bbI}{\ensuremath{\mathbb{I}}\xspace}
\nc{\bbJ}{\ensuremath{\mathbb{J}}\xspace}
\nc{\bbK}{\ensuremath{\mathbb{K}}\xspace}
\nc{\bbL}{\ensuremath{\mathbb{L}}\xspace}
\nc{\bbM}{\ensuremath{\mathbb{M}}\xspace}
\nc{\bbN}{\ensuremath{\mathbb{N}}\xspace}
\nc{\bbO}{\ensuremath{\mathbb{O}}\xspace}
\nc{\bbP}{\ensuremath{\mathbb{P}}\xspace}
\nc{\bbQ}{\ensuremath{\mathbb{Q}}\xspace}
\nc{\bbR}{\ensuremath{\mathbb{R}}\xspace}
\nc{\bbS}{\ensuremath{\mathbb{S}}\xspace}
\nc{\bbT}{\ensuremath{\mathbb{T}}\xspace}
\nc{\bbU}{\ensuremath{\mathbb{U}}\xspace}
\nc{\bbV}{\ensuremath{\mathbb{V}}\xspace}
\nc{\bbW}{\ensuremath{\mathbb{W}}\xspace}
\nc{\bbX}{\ensuremath{\mathbb{X}}\xspace}
\nc{\bbY}{\ensuremath{\mathbb{Y}}\xspace}
\nc{\bbZ}{\ensuremath{\mathbb{Z}}\xspace}

\nc{\mrm}[1]{\ensuremath{\mathrm{#1}}\xspace}
\nc{\mit}[1]{\ensuremath{\mathit{#1}}\xspace}
\nc{\mbf}[1]{\ensuremath{\mathbf{#1}}\xspace}
\nc{\mcal}[1]{\ensuremath{\mathcal{#1}}\xspace}
\nc{\msc}[1]{\ensuremath{\mathscr{#1}}\xspace}
\nc{\mfr}[1]{\ensuremath{\mathfrak{#1}}\xspace}

\nc{\sub}{\subseteq}
\nc{\too}{\longrightarrow}
\nc{\hook}{\hookrightarrow}
\nc{\hooklongrightarrow}{\lhook\joinrel\longrightarrow}
\nc{\hooklong}{\hooklongrightarrow}
\nc{\hooklongleftarrow}{\longleftarrow\joinrel\rhook}
\nc{\twoheadlongrightarrow}{\relbar\joinrel\twoheadrightarrow}
\nc{\longrightleftarrows}{\ \raisebox{0.3ex}{\(\mathrel{\substack{\xrightarrow{\rule{1em}{0em}} \\[-1ex] \xleftarrow{\rule{1em}{0em}}}}\)}\ }

\renc{\ge}{\geqslant}
\renc{\le}{\leqslant}

\nc{\id}{\mathrm{id}}

\DeclareMathOperator{\Hom}{\on{Hom}}
\nc{\uHom}{\underline{\smash{\Hom}}}
\DeclareMathOperator{\Maps}{\on{Maps}}
\DeclareMathOperator{\Aut}{\on{Aut}}
\DeclareMathOperator{\End}{\on{End}}

\nc{\uEnd}{\underline{\smash{\End}}}

\nc{\colim}{\varinjlim}
\renc{\lim}{\varprojlim}
\nc{\Cofib}{\on{Cofib}}
\nc{\Fib}{\on{Fib}}
\nc{\initial}{\varnothing}
\nc{\op}{\mathrm{op}}

\DeclareMathOperator*{\fibprod}{\times}
\DeclareMathOperator*{\fibcoprod}{\operatorname{\sqcup}}

\renc{\setminus}{\smallsetminus}

\usepackage{mathtools}
\DeclarePairedDelimiter\abs{\lvert}{\rvert}

\newcommand{\thmref}[1]{Theorem~\ref{#1}}

\newcommand{\secref}[1]{Sect.~\ref{#1}}

\newcommand{\lemref}[1]{Lemma~\ref{#1}}
\newcommand{\propref}[1]{Proposition~\ref{#1}}
\newcommand{\corref}[1]{Corollary~\ref{#1}}

\newcommand{\remref}[1]{Remark~\ref{#1}}

\renewcommand{\eqref}[1]{(\ref{#1})}

\newcommand{\itemref}[1]{\ref{#1}}

\nc{\A}{\bA}
\nc{\modmod}{/\!\!/}
\renc{\P}{\bP}
\nc{\V}{\bV}
\nc{\Spec}{\on{Spec}}
\nc{\D}{\on{\mbf{D}}}
\nc{\rD}{\on{\mrm{D}}}
\nc{\Dqc}{\on{\mbf{D}}_{\mrm{qc}}}
\nc{\bDelta}{\mathbf{\Delta}}
\nc{\Cech}{\textnormal{\v{C}}}
\nc{\Dperf}{\on{\mbf{D}}_{\mrm{perf}}}
\nc{\Coh}{\on{Coh}}
\nc{\Qcoh}{\on{Qcoh}}
\nc{\Dcoh}{\on{\mbf{D}}_{\mrm{coh}}}
\nc{\cl}{{\mrm{cl}}}
\nc{\Bl}{\on{Bl}}
\nc{\vir}{\mrm{vir}}
\nc{\Zar}{\mrm{Zar}}
\nc{\et}{\mrm{\acute{e}t}}
\nc{\Nis}{\mrm{Nis}}
\renc{\H}{\on{H}}
\nc{\BM}{\mrm{BM}}
\nc{\Z}{\bZ}
\nc{\Q}{\bQ}
\nc{\K}{{\on{K}}}
\nc{\G}{{\on{G}}}
\nc{\KH}{\mrm{KH}}
\nc{\KGL}{\mrm{KGL}}
\nc{\rightrightrightarrows}
  {\mathrel{\substack{\textstyle\rightarrow\\[-0.6ex]
   \textstyle\rightarrow \\[-0.6ex]
   \textstyle\rightarrow}}}
\nc{\rightrightrightrightarrows}
  {\mathrel{\substack{\textstyle\rightarrow\\[-0.6ex]
   \textstyle\rightarrow \\[-0.6ex]
   \textstyle\rightarrow \\[-0.6ex]
   \textstyle\rightarrow}}}
\nc{\Einfty}{{\sE_\infty}}
\renc{\sp}{\mrm{sp}}
\nc{\Td}{\on{Td}}
\nc{\ch}{\on{ch}}
\nc{\RGamma}{R\Gamma}
\nc{\red}{\mrm{red}}
\nc{\der}{{\mrm{der}}}
\nc{\Mod}{{\mrm{Mod}}}
\nc{\Gr}{{\on{Gr}}}
\nc{\Ind}{\on{Ind}}
\nc{\Pro}{\on{Pro}}
\nc{\dash}{{\textnormal{-}}}
\nc{\Cat}{\infty\dash\mrm{Cat}}
\nc{\Pres}{\mrm{Pres}}
\nc{\form}{\widehat}
\nc{\R}{\bR}
\renc{\L}{\bL}
\nc{\otimesL}{\mathchoice{\overset{\bL}{\otimes}}{\otimes^\bL}{\otimes^\bL}{\otimes^\bL}}
\nc{\fibprodR}{\fibprod^\bR}
\nc{\uRHom}{\bR\uHom}
\nc{\GL}{\mrm{GL}}
\nc{\SW}{\on{SW}}
\nc{\Vect}{\on{Vect}}
\nc{\Fun}{\on{Fun}}
\nc{\Nat}{\on{Nat}}
\nc{\un}{\mbf{1}}
\nc{\pr}{\mrm{pr}}
\nc{\pt}{\mrm{pt}}
\nc{\vb}[1]{\langle{#1}\rangle}
\nc{\Pt}{\on{Pt}}
\nc{\lisse}{{\triangleleft}}
\nc{\Lis}{\mrm{Lis}}
\nc{\LisStk}{\mrm{LisStk}}
\nc{\Et}{{\mrm{Et}}}
\nc{\aff}{\mrm{aff}}
\nc{\qproj}{\mrm{qproj}}
\nc{\fp}{\mrm{fp}}
\nc{\ft}{\mrm{ft}}
\nc{\affft}{\mrm{affft}}
\nc{\sm}{\mrm{sm}}
\nc{\lci}{\mrm{lci}}
\nc{\Lisftsm}{\Lis^{\ft:\sm}}
\nc{\Lisaffftsm}{\Lis^{\affft:\sm}}
\nc{\Lisaspftsm}{\Lis^{\mrm{aspft}:\sm}}
\renc{\top}{\mrm{top}}
\nc{\C}{\on{C}}
\nc{\Chom}{\mrm{C}_\bullet}
\nc{\Ccoh}{\mrm{C}^\bullet}
\nc{\Ccohc}{\mrm{C}_{\mrm{c}}^\bullet}
\nc{\CBM}{\mrm{C}^{\BM}_\bullet}
\nc{\mot}{\mrm{mot}}
\nc{\Chommot}{\mrm{C}^{\mot}_\bullet}
\nc{\Top}{\mrm{Top}}
\renc{\top}{\mrm{top}}
\nc{\Spc}{\mrm{Spc}}
\nc{\Stk}{\mrm{Stk}}
\nc{\Art}{\mrm{Art}}
\nc{\Shv}{\on{Shv}}
\nc{\Spt}{\mrm{Spt}}
\nc{\heart}{\heartsuit}
\nc{\an}{\mrm{an}}
\nc{\Anima}{\mrm{Ani}}
\nc{\Aff}{\mrm{Aff}}
\nc{\MotSpc}{{\mrm{MAni}}}
\renc{\L}{\mrm{\bL}}
\nc{\h}{\mrm{h}}
\nc{\Sm}{\mrm{Sm}}
\nc{\Sch}{\mrm{Sch}}
\nc{\Asp}{\mrm{Asp}}
\nc{\Betti}{\mrm{Bet}}
\nc{\cdh}{\mrm{cdh}}
\renc{\Re}{\mrm{Re}}
\nc{\bz}{\mathbf{z}}
\nc{\Tot}{\on{Tot}}
\nc{\MGL}{\mrm{MGL}}
\nc{\inv}[1]{[\tfrac{1}{#1}]}
\nc{\SH}{\on{\mathbf{SH}}}
\nc{\DM}{\on{\mathbf{DM}}}
\nc{\uPic}{\underline{\on{Pic}}}
\nc{\Cmot}{\on{C}_{\mrm{mot}}^\bullet}
\nc{\Hmot}{\on{H}_{\mrm{mot}}}

\nc{\scr}{\term{derived commutative ring}}
\nc{\scrs}{\term{derived commutative rings}}

\nc{\inftyCat}{\term{$\infty$-category}}
\nc{\inftyCats}{\term{$\infty$-categories}}

\nc{\inftyGrpd}{\term{$\infty$-groupoid}}
\nc{\inftyGrpds}{\term{$\infty$-groupoids}}

\nc{\dA}{\term{derived Artin}}

\nc{\BSr}{\hyperref[cond:BSr]{$(\mathrm{BS}_{r})$}\xspace}

\title{Descendability and descent in topological weaves\vspace{-2mm}}
\author[A.\,A. Khan]{Adeel A. Khan}
\date{2026-07-16}

\makeatletter
\def\l@subsection{\@tocline{2}{0pt}{4pc}{6pc}{}}
\makeatother

\begin{document}

\begin{abstract}
  We prove a criterion for a finitely presented surjection of algebraic spaces to be descendable in a topological weave.
  We apply this to show that étale motivic spectra satisfy $h$-descent on noetherian finite-dimensional schemes with residue fields of uniformly bounded \'etale cohomological dimension.
  We also show that rational motivic cohomology satisfies arc-descent in weights $\le 1$, and we construct the ``forgetting supports'' isomorphism $f_! \simeq f_*$ for a proper DM-type morphism of Artin stacks, in rational motivic sheaves.
  \vspace{-5mm}
\end{abstract}

\maketitle

\renewcommand\contentsname{\vspace{-1cm}}
\tableofcontents

\setlength{\parindent}{1em}
\parskip 0.6em
\raggedbottom

\thispagestyle{empty}

\changelocaltocdepth{1}

\section*{Introduction}

A fundamental result in the theory of \'etale cohomology is the cohomological descent theorem of \cite[Exp.~V\textsuperscript{bis}, Prop.~4.3.2]{SGA4}, which says that the presheaf $S \mapsto R\Gamma_{\et}(S; \Lambda)$, where $\Lambda$ is a torsion commutative ring, satisfies descent for the $h$-topology of Voevodsky \cite{VoevodskyH}; this amounts to descent for both \'etale surjections and finitely presented proper surjections.
Moreover, this can be upgraded to $h$-descent for the presheaf of $\infty$-categories $S \mapsto \on{D}_\et(S; \Lambda)$, where the transition functors are $*$-pullbacks.

In this paper we study the question of $h$-descent for abstract topologically flavoured six-functor formalisms, i.e., for the topological weaves of \cite{weaves}.
We provide a general criterion for $h$-descent, which we apply in particular to various flavours of motivic sheaves.
Consider first Morel's model for rational motivic sheaves: over an arbitrary base $S$, $\SH_{\Q,+}(S)$ denotes the plus part of the $\infty$-category of rational motivic spectra over $S$ (see \cite[\S 16]{CisinskiDegliseBook}).

\begin{thmX}\label{thmX:DMQ}
	$\SH_{\Q,+}(-)$ satisfies $h$-descent on all algebraic spaces.
\end{thmX}

Next consider $\SH_{\et}(S)$, the $\infty$-category of motivic spectra over $S$ satisfying étale hyperdescent, and (for $S$ noetherian of finite dimension) the $\infty$-category $\DM_{\et}(S; \Lambda)$ of étale motivic sheaves as defined in \cite[\S 2]{CDEtale}.

\begin{thmX}\label{thmX:SHet}
	$\SH_{\et}(-)$ satisfies $h$-descent on noetherian finite-dimensional algebraic spaces whose residue fields have uniformly bounded étale cohomological dimension.
	The same holds for $\DM_{\et}(-; \Lambda)$ for any commutative ring $\Lambda$.
\end{thmX}

The cases of $\SH_{\Q,+}$ and $\DM_{\et}(-; \Lambda)$ are folklore.\footnote{%
  For $\SH_{\Q,+}$ one may use Nisnevich descent and continuity to reduce to the case of \emph{noetherian finite-dimensional} schemes, where it can be easily derived from the work done in \cite{CisinskiDegliseBook} using the $\infty$-categorical techniques developed in \cite{LurieHTT,LurieHA}.
  We will give a more direct argument.
  For $\DM_{\et}(-; \Lambda)$, it can similarly be derived from the work of Ayoub and Cisinski--D\'eglise \cite{Ayoub,CDEtale}.
}
On global sections, \thmref{thmX:DMQ} recovers $h$-descent for motivic cohomology with rational coefficients, recorded in the literature on quasi-excellent noetherian finite-dimensional schemes by Cisinski--D\'eglise (see \cite[Thms.~14.3.4, 16.2.13]{CisinskiDegliseBook}).
We will also extend this to hyperdescent for the \emph{arc}-topology of \cite{BhattMathew}, for rational motivic cohomology in weights $\le 1$ (and conditionally in arbitrary weights); see \thmref{thm:BS arc}.
\thmref{thmX:SHet} recovers in particular $h$-descent for {\'etale} motivic cohomology, which was recorded for noetherian finite-dimensional schemes in \cite[Cor.~5.5.7]{CDEtale}.

In fact, for both $\SH_{\Q,+}$ and $\SH_{\et}$ we show the stronger assertion that any finitely presented surjection $f : Y \twoheadrightarrow X$ of qcqs algebraic spaces is \emph{descendable} in the sense of Mathew \cite{Mathew}: that is, the unit object $\un_X \in \D(X)$ can be built out of $f_*(\un_Y)$ as follows: it lies in the smallest full subcategory containing $f_*(\un_Y)$ and closed under cofibres, fibres, direct summands, and $(-) \otimes \sF$ for every $\sF \in \D(X)$.
When $f$ is proper, this in particular implies \v{C}ech descent along $f$ (see \propref{prop:proper descendable implies descent}).

Switching our focus from descent to descendability has several advantages.
First, it implies not only that every $\sF \in \D(X)$ can be expressed as the totalization $\Tot(f_{\bullet,*}f_\bullet^*(\sF))$, where $f_\bullet: Y_\bullet \to X$ is the \v{C}ech nerve, but also that this totalization commutes past any exact functor $F$ (see \remref{rem:weave totalization commutes}):
\[
  F(\sF) \simeq \Tot\bigl( F\bigl( f_{\bullet,*}f_\bullet^*(\sF) \bigr) \bigr).
\]

Second, descendability is preserved by any symmetric monoidal functor of stable \inftyCats.
\thmref{thmX:DMQ} therefore implies that any $\Q$-linear oriented topological weave satisfies $h$-descent on algebraic spaces; similarly \thmref{thmX:SHet} implies that a topological weave satisfies $h$-descent if and only if it satisfies étale descent, over bases with residue fields of uniformly bounded étale cohomological dimension.

The most important advantage of descendability in \emph{proving} Theorems~\ref{thmX:DMQ} and \ref{thmX:SHet} is that it admits the following useful criterion.

\begin{thmX}\label{thmX:bootstrap}
  Let $\D$ be a topological weave.
  Then the following conditions are equivalent:
  \begin{thmlist}
    \item
    Every finite étale surjection and every finite radicial surjection of qcqs algebraic spaces is descendable.

    \item
    Every finitely presented surjection of qcqs algebraic spaces is descendable.
  \end{thmlist}
\end{thmX}

We end the paper with the following application, which shows how the stronger form of \v{C}ech descent provided by descendability can be useful.

\begin{thmX}\label{thmX:fsupp}
  Let $f : \sX \to \sY$ be a proper $1$-Artin morphism with proper diagonal.
  Assume that $f$ is representable by Deligne--Mumford stacks, or that $f$ has tame relative inertia and $\sY$ is equicharacteristic.
  Then the canonical \emph{forgetting supports} map $\alpha_f : f_!(\sF) \to f_*(\sF)$ is invertible for every $\sF \in \SH_{\Q,+}(\sX)$.
\end{thmX}

A similar result was claimed in \cite[Thm.~A.7]{virtual}, with the following ``proof'' (we consider the DM variant).
One can reduce to the case where $\sY$ is a scheme and $\sX$ is a DM stack.
It is known that one can find a proper representable surjection $g : Z \twoheadrightarrow \sX$ where $Z$ is a scheme, and then \v{C}ech descent gives $\sF \simeq \Tot( g_{\bullet,*}g_\bullet^*(\sF) )$ where $g_\bullet : Z_\bullet \to \sX$ is the \v{C}ech nerve.
By functoriality of $\alpha$ this reduces us to the invertibility of $\alpha_{g_n}$ and $\alpha_{f \circ g_n}$.
The gap here is that $f_!$ might not \textit{a priori} commute past the totalization.
This is precisely what is patched by the observation that $g$ is \emph{descendable} (\thmref{thm:descendability DM}).

\subsection{Acknowledgments}

I thank Tom Bachmann for helping me avoid inverting $2$ in \thmref{thmX:SHet}, Denis-Charles Cisinski for introducing me to the question of $h$-descent for $\SH_\et$, and Lucas Mann for explanations about descendability.
I further thank Bachmann, Cisinski, Niklas Kipp, Charanya Ravi, and Swann Tubach for comments on previous drafts.

I acknowledge support from the grants AS-CDA-112-M01 (Academia Sinica) and NSTC 112-2628-M-001-006.
I am grateful to the RIMS at Kyoto University for its hospitality.

\changelocaltocdepth{2}

\section{Descendability}

\subsection{Definitions}

We briefly review the notion of \emph{descendability} as studied by Mathew \cite[\S 3.3]{Mathew}; proofs of the following statements can be found either there or in \cite[\S 11.2]{BhattScholze}.
Let $\sD$ be a symmetric monoidal presentable stable \inftyCat, in which the tensor product is colimit-preserving in each argument.

\begin{defn}\label{defn:desc}
  A morphism $A \to B$ of $E_\infty$-algebras in $\sD$ is \emph{descendable} if the morphism $\{A\} \to \{\Tot_{\le n} (B^\bullet)\}_n$ is invertible in $\on{Pro}(\sD)$.
  Here, $B^\bullet$ denotes the \v{C}ech nerve of $A \to B$, the cosimplicial object whose $n$th term is the $(n+1)$-fold tensor product of $B$ over $A$, and $\{\Tot_{\le n} (B^\bullet) \}_n$ is its tower of partial totalizations, regarded as a pro-object of $\sD$.
\end{defn}

Descendability is closed under composition and extension of scalars.
Moreover, any symmetric monoidal functor preserves descendability.
See \cite[Cor.~3.21, Prop.~3.24]{Mathew}.
We will also say that an $E_\infty$-algebra $A$ is \emph{descendable} if the unit $\un \to A$ is descendable.

\begin{rem}\label{rem:descendable commute with limits}
  If $A \to B$ is descendable, then the pro-object $\{\Tot_{\le n} (B^\bullet)\}_n$ is essentially constant; in particular, we have (see e.g. \cite[Prop.~3.10]{Mathew}):
  \[ F \bigl(\Tot (B^\bullet) \bigr) \simeq \lim_n F \bigl( \Tot_{\le n} (B^\bullet) \bigr) \simeq \lim_n \Tot_{\le n} \bigl( F(B^\bullet) \bigr) \simeq \Tot \bigl( F(B^\bullet) \bigr) \]
  for any left-exact functor $F : \sD \to \sV$ valued in an \inftyCat with limits.
\end{rem}

We have the following criterion for descendability:

\begin{lem}
	Let $A \to B$ be a morphism of $E_\infty$-algebras in $\sD$.
	Let $\sC \sub \Mod_{A}(\sD)$ be the smallest thick subcategory which contains $B$ and is closed under $(-) \otimes M$ for every $M \in \Mod_{A}(\sD)$.
  Then $A \to B$ is descendable if and only if $\sC$ contains $A$.
\end{lem}

We also have the following quantitative version:

\begin{defn}\label{defn:desc index}
  Let $A \to B$ be a morphism of $E_\infty$-algebras in $\sD$.
  We say $A \to B$ is \emph{descendable of index $\le N$}, for a natural number $N > 0$, if the canonical map $A \to \Tot_{\le N-1}(B^\bullet)$ admits a retraction in $\sD$.
\end{defn}

Since $\Tot_{\le 0}(B^\bullet) \simeq B$, descendability of index $\le 1$ is the condition that the unit $A \to B$ admits a retraction.
The index also admits the following description (see \cite[Prop.~4.9 (4)]{MNNNilpotence}).

\begin{lem}\label{lem:index augmentation ideal}
  Let $A \to B$ be a morphism of $E_\infty$-algebras in $\sD$, and let $I$ denote its fibre in $\sD$.
  Then $A \to B$ is descendable of index $\le N$ if and only if the map $\varepsilon^{N} : I^{\otimes_A N} \to A^{\otimes_A N} \simeq A$ is null-homotopic.
\end{lem}

The connection between Definitions~\ref{defn:desc} and \ref{defn:desc index} is as follows (see \cite[Lem.~11.20]{BhattScholze}).

\begin{lem}\label{lem:desc of some index}
  Let $A \to B$ be a morphism of $E_\infty$-algebras in $\sD$.
  Then $A \to B$ is descendable if and only if it is descendable of index $\le N$ for some $N$.
\end{lem}

\begin{lem}\label{lem:monotone index}
  Let $A \to B$ be a morphism of $E_\infty$-algebras in $\sD$.
	If $B$ is descendable of index $\le N$, then so is $A$.
\end{lem}
\begin{proof}
  There is an induced morphism $I_A \to I_B$ on fibres, and $\varepsilon_A^{N}$ factors as
	\[ \varepsilon_A^{N} : I_A^{\otimes N} \to I_B^{\otimes N} \xrightarrow{\varepsilon_B^{N}} \un. \qedhere \]
\end{proof}

The following lemma shows that descendability is ``descendable-local''.

\begin{lem}\label{lem:descendable local abstract}
  Let $A$ and $B$ be $E_\infty$-algebras in $\sD$.
  If $B$ is descendable of index $\le N$ and $B \to A \otimes B$ is descendable of index $\le M$, then $A$ is descendable of index $\le NM$.
\end{lem}
\begin{proof}
	Let $I_A$ and $I_B$ denote the fibres of the unit maps of $A$ and $B$, respectively.
	We have the diagram with exact rows
	\[\begin{tikzcd}
    I_A \ar{r}{\varepsilon_A}\ar{d}
    & \un \ar{r}{\eta_A}\ar{d}{\eta_B}
    & A \ar{d}
    \\
    I_A \otimes B \ar{r}{\varepsilon_A \otimes \id_B}
    & B \ar{r}
    & A \otimes B.
	\end{tikzcd}\]
	By assumption, the map $\varepsilon_A^M \otimes \id_B : I_A^{\otimes M} \otimes B \to B$ is null-homotopic, hence so is $\eta_B \circ \varepsilon_A^M : I_A^{\otimes M} \to B$.
	Thus from the exact triangle $I_B \to \un \to B$, there exists a lift $\phi : I_A^{\otimes M} \to I_B$ with $\varepsilon_B \circ \phi \simeq \varepsilon_A^M : I_A^{\otimes M} \to \un$.
	We then have
	\[
    \varepsilon_A^{NM} \simeq
    (\varepsilon_A^M)^{\otimes N} \simeq
    (\varepsilon_B \circ \phi)^{\otimes N} \simeq
    \varepsilon_B^N \circ \phi^{\otimes N} \simeq
    0,
	\]
	as claimed.
\end{proof}

For a presentable weave $\D$, we define:

\begin{defn}
  A morphism of algebraic spaces $f : Y \to X$ is \emph{descendable} (of index $\le N$) in the weave $\D$ if $\un_X \to f_*(\un_Y)$ is descendable (of index $\le N$), as a morphism of $E_\infty$-algebras in $\D(X)$.
\end{defn}

This property is closed under composition and base change.

\begin{cor}\label{cor:descendable local}
	Let $f : Y \to X$ and $p : X' \to X$ be morphisms, and let $f' : Y' \to X'$ denote the base change of $f$ along $p$.
	Suppose that $p_*$ satisfies the projection formula (Pr2) (see \cite[Def.~2.18]{weaves}), and the exchange morphism $\mrm{Ex}^*_* : p^*f_* \un_Y \to f'_{*} \un_{Y'}$ is invertible.
	If $p$ is descendable of index $\le N$ and $f'$ is descendable of index $\le M$, then $f$ is descendable of index $\le NM$.
\end{cor}
\begin{proof}
  Let $A = f_*\un_Y$ and $B = p_*\un_{X'}$; by the assumptions, $A \otimes B \simeq p_*(f'_*\un_{Y'})$.
  Hence \lemref{lem:descendable local abstract} shows that $B \to A \otimes B$ is descendable of index $\le M$ since $f'$ is.
\end{proof}

\subsection{Criteria for finite étale covers}
\label{ssec:etale descent}

Let $f : Y \to X$ be a finite étale morphism of algebraic spaces.
In order for $f$ to be descendable of index $\le 1$, it suffices that the composite
\begin{equation}\label{eq:fetcons}
  \un_X
  \xrightarrow{\mrm{unit}} f_*f^*(\un_X) \simeq f_!f^!(\un_X)
  \xrightarrow{\mrm{counit}} \un_X,
\end{equation}
be invertible.
(This is the Euler characteristic of the dualizable object $f_*(\un_Y)$.)
This condition can be made more concrete using the following fact from \cite[Prop.~2.2.5]{topinv} (see also \cite[Rem.~3.2.2]{topinv}):

\begin{lem}\label{lem:fet chi field}
  Let $\D$ be a topological weave.
	Let $X = \Spec(\kappa)$ be the spectrum of a field and $f : Y \to X$ a finite étale morphism of degree $d$.
	Then the Euler characteristic \eqref{eq:fetcons} is the class $d_\varepsilon \in \End_{\D(k(x))}(\un_{k(x)})$.
\end{lem}

We say that $\D$ satisfies \emph{conservativity of stalks} for an algebraic space $X$ if the functors $x^* : \D(X) \to \D(k(x))$, $x \in X$, are jointly conservative.
For example, this holds for $X$ locally noetherian when $\D$ is a topological weave satisfying continuity and geometric generation (see \cite[Prop.~3.13]{weaves}).

\begin{cor}\label{cor:fet chi stalkwise}
  Let $\D$ be a topological weave satisfying conservativity of stalks for an algebraic space $X$.
	Let $f : Y \to X$ be a finite étale morphism of degree $d$.
	Then the following conditions are equivalent:
	\begin{thmlist}
	  \item The Euler characteristic \eqref{eq:fetcons} is invertible.
		\item The class $d_\varepsilon \in \End_{\D(X)}(\un_{X})$ is invertible.
		\item For every point $x \in X$, the class $d_\varepsilon \in \End_{\D(k(x))}(\un_{k(x)})$ is invertible.
  \end{thmlist}
\end{cor}
\begin{proof}
	Since \eqref{eq:fetcons} and $d_\varepsilon$ are both functorial in $X$, the first two claims are both equivalent to the third by conservativity of stalks.
\end{proof}

Recall that an \emph{orientation} of $\D$ is a factorization of the Thom twist map $\K \to \uPic$ through the rank map $\K \to \Z$.
By construction of $d_\varepsilon$, such an orientation identifies $d_\varepsilon \simeq d$ in $\End_{\D(X)}(\un_X)$ for every $X$.
In particular:

\begin{cor}\label{cor:fet chi oriented}
  Let $\D$ be an oriented topological weave satisfying conservativity of stalks for an algebraic space $X$.
	Let $f : Y \to X$ be a finite étale morphism of degree $d$.
	Then the Euler characteristic \eqref{eq:fetcons} is invertible if and only if the class $d \in \End_{\D(X)}(\un_{X})$ is invertible.
\end{cor}

\begin{cor}\label{cor:fet chi rational}
  Let $\D$ be an oriented $\bQ$-linear topological weave satisfying conservativity of stalks for an algebraic space $X$.
  Then for every finite étale surjection $f : Y \twoheadrightarrow X$, the Euler characteristic \eqref{eq:fetcons} is invertible.
\end{cor}

\begin{cor}\label{cor:fet chi approx}
  Let $\D$ be a topological weave satisfying continuity and geometric generation.
  Let $f : Y \to X$ be a finite étale morphism of constant degree $d$ between qcqs algebraic spaces over a noetherian algebraic space $S$.
  If the class $d_\varepsilon \in \End_{\D(\kappa)}(\un_\kappa)$ is invertible for every field $\kappa$ over $S$, then the Euler characteristic \eqref{eq:fetcons} of $f$ is invertible; in particular, $f$ is descendable of index $\le 1$.
\end{cor}
\begin{proof}
  When $X$ is locally noetherian, $\D$ satisfies conservativity of stalks for $X$ (see \cite[Prop.~3.13]{weaves}), so the claim follows from \corref{cor:fet chi stalkwise}.
  In general, write $X$ as the limit of a cofiltered system $(X_\alpha)_\alpha$ of algebraic spaces of finite presentation over $S$, with affine transition maps (see \cite[Thm.~D]{RydhApprox}).
  By \cite[Props.~B.2, B.3]{RydhApprox} $f$ is the base change of some finite étale $f_\alpha : Y_\alpha \to X_\alpha$ along $u_\alpha : X \to X_\alpha$ (for some $\alpha$).
  After replacing $X_\alpha$ by the closed and open subspace over which $f_\alpha$ is of degree $d$ (which contains the image of $u_\alpha$), we may moreover assume that $f_\alpha$ is of constant degree $d$.
  Since $X_\alpha$ is noetherian, the Euler characteristic of $f_\alpha$ is invertible by the previous case, hence so is that of its base change $f$.
\end{proof}

\begin{cor}\label{cor:fet chi qcqs}
  Let $\D$ be an oriented $\bQ$-linear topological weave satisfying continuity and geometric generation.
  Then for every finite étale surjection $f : Y \twoheadrightarrow X$ of qcqs algebraic spaces, the Euler characteristic \eqref{eq:fetcons} is invertible; in particular, $f$ is descendable of index $\le 1$.
\end{cor}
\begin{proof}
  By radditivity, we may assume that $f$ is of constant degree $d>0$.
  As in \corref{cor:fet chi oriented}, the orientation identifies $d_\varepsilon \simeq d$ in $\End_{\D(\kappa)}(\un_\kappa)$ for every field $\kappa$, and $d$ is invertible by $\bQ$-linearity.
  The claim now follows from \corref{cor:fet chi approx}, applied with $B = \Spec(\bZ)$.
\end{proof}

\begin{lem}\label{lem:fet desc torsor}
  Let $\D$ be a radditive weave.
  Then every finite étale surjection is descendable (of index $\le N$) if and only if for every finite group $G$, every $G$-torsor is descendable (of index $\le N$).
\end{lem}
\begin{proof}
  Let $f : Y \twoheadrightarrow X$ be a finite étale surjection of degree $d$; by radditivity we may assume that $d$ is constant.
  The permutation action of the symmetric group $S_d$ on $Y^d := Y^{\fibprod_X d}$ is free and transitive away from the big diagonal
  \[ \Delta^d := \bigcup_{1 \le i<j \le d} \pr_{i,j}^{-1}(Y) \;\sub\; Y^d, \]
  where $\pr_{i,j} : Y^d \to Y \fibprod_X Y$ are the projections, and $Y$ is regarded as a subscheme of $Y \fibprod_X Y$ via the (open and closed) diagonal $Y \to Y \fibprod_X Y$.
  Thus the map
  \[
    g : Y^{(d)} := Y^d \setminus \Delta^d \sub Y^d \longrightarrow X
  \]
  is an $S_d$-torsor.
  It can be expressed equivalently as the composite
  \[
    g : Y^{(d)} \xrightarrow{p} Y \xrightarrow{f} X,
  \]
  where $p$ is the restriction of the first projection of $Y^d$.
  The unit of $p$ gives rise to a morphism of $E_\infty$-algebras
  \[ f_*\un_Y \to f_*p_*\un_{Y^{(d)}} \simeq g_*\un_{Y^{(d)}}. \]
  Now \lemref{lem:monotone index} shows that if $g$ is descendable of index $\le N$, then so is $f$.
\end{proof}

\subsection{Criteria for finite radicial covers}
\label{ssec:top inv}

\begin{lem}\label{lem:frad}
  Let $\D$ be a weave satisfying nil-invariance.
  Then for every finite radicial surjection $f : Y \twoheadrightarrow X$, the following conditions are equivalent:
  \begin{thmlist}
    \item The functor $f^*$ is conservative.
    \item The functor $f^*$ is an equivalence.
  \end{thmlist}
  Moreover, in this case there is a canonical isomorphism $f^* \simeq f^!$.
\end{lem}
\begin{proof}
  The diagonal $\Delta : Y \to Y\fibprod_X Y$ is a nil-immersion.
  By nil-invariance, we have that $\Delta^* \simeq \Delta^!$ is an equivalence.
  From this it follows that the counit $f^*f_* \to \id$ is invertible (see proof of \cite[Prop.~2.2.1]{topinv} for details).
  Then by the triangle identities, the unit $\id \to f_*f^*$ is invertible if and only if $f^*$ is conservative.
  When this holds, i.e. $f^*$ is an equivalence, it follows formally that its right adjoint $f_* \simeq f_!$ is also an equivalence, and that its left and right adjoints $f^*$ and $f^!$ are tautologically identified.
\end{proof}

We say that a weave $\D$ satisfies \emph{topological invariance} if it satisfies nil-invariance and, for every finite radicial surjection $f : Y \twoheadrightarrow X$, the functor $f^*$ is an equivalence.
By \cite[Thm.~2.1.1, Rem.~2.2.9]{topinv}, we have:

\begin{thm}\label{thm:topinv}
	Let $\D$ be a topological weave and $f : Y \twoheadrightarrow X$  a finite radicial surjection.
  If every prime number is invertible either in $\sO_X$ or in $\End_{\D(X)}(\un_X)$, then $f^*$ is an equivalence.
  In particular, $\D[\sP^{-1}]$ satisfies topological invariance on the category of algebraic spaces over $X$, where $\sP$ is the set of primes not invertible in $\sO_X$.
\end{thm}
\begin{proof}
  Let $\sQ$ be the set of primes invertible in $\End_{\D(X)}(\un_X)$.
  Every $q \in \sQ$ acts invertibly on $\un_X$, hence on $\D(X)$.
  Since $\SH$ is the universal topological weave, there exists a morphism of weaves $\SH \to \D$ which factors through $\SH[\sQ^{-1}]$ on algebraic spaces over $X$.
  Thus the invertibility of the unit and counit of the adjunction $(f^*, f_* \simeq f_!)$ follows from the case of $\SH[\sQ^{-1}]$, which is \cite[Thm.~2.1.1]{topinv}.
\end{proof}

\subsection{Bootstrapping}

Our main result in this section is as follows:

\begin{thm}\label{thm:descendability}
  Let $\D$ be a weave satisfying the localization property.
  Then the following conditions are equivalent:
  \begin{thmlist}
    \item\label{item:desc/blocks}
    Every finite étale surjection and every finite radicial surjection of qcqs algebraic spaces is descendable.

    \item\label{item:desc/all}
    Every finitely presented surjection $f : Y \twoheadrightarrow X$ of qcqs algebraic spaces is descendable.
  \end{thmlist}
\end{thm}

\subsection{Proof of \thmref{thm:descendability}}

\begin{lem}\label{lem:descendability section}
	Let $f : Y \to X$ be a morphism admitting a section $s$.
	Then $f$ is descendable of index $\le 1$.
\end{lem}
\begin{proof}
	The unit map of $s_*(\un_X)$ induces a retraction $f_*\un_Y \to f_*s_* \un_X \simeq \un_X$ of the unit of $f_*(\un_Y)$.
\end{proof}

\begin{lem}\label{lem:descendability and localization}
  Let $\D$ be a weave satisfying the localization property and $f : Y \to X$ a morphism.
  Given a closed immersion $i : Z \hook X$ with open complement $j : U \hook X$, let $f_Z$ and $f_U$ denote the base changes of $f$ to $Z$ and $U$.
  If $f_Z$ and $f_U$ are descendable of index $\le N_Z$ and $\le N_U$, respectively, then $f$ is descendable of index $\le N_Z+N_U$.
\end{lem}
\begin{proof}
  By smooth base change, we have $j^*\sI_f \simeq \sI_{f_U}$.
  Consider the localization triangle $i_*i^!\un_X \to \un_X \to j_*\un_U$.
  The morphism $\varepsilon_f^{N_U}$ factors as
  \[
    \varepsilon_f^{N_U} : \sI_f^{\otimes N_U} \xrightarrow{\phi_f} i_*i^!\un_X \to \un_X.
  \]
  Indeed, its composite with the unit $\un_X \to j_*\un_U$ factors through $j_*$ of $j^* \varepsilon_f^{N_U} \simeq \varepsilon_{f_U}^{N_U} \simeq 0$, hence is null-homotopic.
  Thus the map $\varepsilon_f^{N_U+N_Z}$ factors through
  \[
    \sI_f^{\otimes N_Z+N_U}
    \simeq \sI_f^{\otimes N_Z} \otimes \sI_f^{\otimes N_U}
    \xrightarrow{\id \otimes \phi_f} \sI_f^{\otimes N_Z} \otimes i_*i^!(\un_X)
    \simeq i_*(i^* \sI_{f}^{\otimes N_Z} \otimes i^! \un_X)
  \]
  and
  \[
    i_*(i^*\sI_{f}^{\otimes N_Z} \otimes i^!\un_X)
    \xrightarrow{i_*(i^*\varepsilon_{f}^{N_Z} \otimes \id)} i_*(i^*\un_X \otimes i^!\un_X)
    \simeq i_*i^!(\un_X)
    \xrightarrow{\mrm{counit}} \un_X.
  \]
  Since $f_Z$ is descendable of index $\le N_Z$, applying \lemref{lem:monotone index} to the morphism of algebras $\mrm{Ex}^*_* : i^*f_* \un_Y \to f_{Z,*} \un_{Y_Z}$ yields that $i^* \varepsilon_{f}^{N_Z}$ is null-homotopic.
  The claim follows.
\end{proof}

\begin{proof}[Proof of \thmref{thm:descendability}]
  Since finite étale and finite radicial surjections are in particular finitely presented surjections, \itemref{item:desc/all} implies \itemref{item:desc/blocks}.
  Conversely, assume \itemref{item:desc/blocks} and let $f : Y \twoheadrightarrow X$ be a finitely presented surjection.
  Suppose first that $X$ is a qcqs scheme.
  Then there exists by \cite[IV\textsubscript{4}, Prop.~17.16.4]{EGA}\footnote{%
    stated for morphisms of schemes, but we may apply it to the composite $Y' \twoheadrightarrow Y \twoheadrightarrow X$ where $Y' \twoheadrightarrow Y$ is an étale surjection from a scheme
  } a stratification by locally closed subschemes $X_\alpha$ for which there are finite radicial and finite étale surjections, $g_\alpha : X'_\alpha \to X_\alpha$ and $h_\alpha : X''_\alpha \to X'_\alpha$, respectively, such that $f$ admits a section after base change along each $X''_\alpha \to X$.

  By \lemref{lem:descendability section}, each $f \times_X X''_\alpha$ is descendable.
  Since $g_\alpha$ and $h_\alpha$ are descendable by assumption, \corref{cor:descendable local}, applied first along $h_\alpha$ and then along $g_\alpha$, implies that each $f \times_X X_\alpha$ is descendable.
  Refining the stratification $(X_\alpha)_\alpha$ if necessary, we may choose a filtration by closed subsets $\initial = Z_0 \sub Z_1 \sub \cdots \sub Z_n = X$ such that each $Z_k \setminus Z_{k-1}$ is a disjoint union of strata $X_\alpha$.
  Since the weave is radditive, $f \times_X (Z_k\setminus Z_{k-1})$ is descendable, and applying \lemref{lem:descendability and localization} recursively shows that $f$ is descendable.

  For a general algebraic space $X$, there exists by \cite[Ex.~3.2, Prop.~3.3]{HallRydhAddendum} a Nisnevich covering $p : X' \twoheadrightarrow X$ with $X'$ a qcqs scheme and a monomorphic splitting sequence $\initial = V_0 \sub V_1 \sub \cdots \sub V_m = X$ where each $V_j$ is a quasi-compact open and $Z_j := (V_j \setminus V_{j-1})_{\red}$ are schemes.
  By the case above, each base change $f \times_X Z_j$ is descendable.
  Applying \lemref{lem:descendability and localization} recursively to the closed subset $Z_j \sub V_j$ with open complement $V_{j-1}$, we deduce that $f$ is descendable.
\end{proof}

\begin{rem}\label{rem:descendability index bound}
  In the context of \thmref{thm:descendability}, suppose additionally that the finite étale and finite radicial surjections in \itemref{item:desc/blocks} are descendable \emph{of index $\le 1$}.
  In this case the proof yields the following uniform bound:
  \begin{enumerate}
    \item[$(\ast)$]
    For every qcqs algebraic space $X$ of finite Krull dimension there exists a natural number $N_X$ (depending only on $X$) such that every finitely presented surjection $f : Y \twoheadrightarrow X$ is descendable of index $\le N_X$.
  \end{enumerate}
  If moreover $X$ is a qcqs scheme of Krull dimension $d_X$, one may choose the filtration $(Z_k)_k$ appearing in the proof to be the filtration by dimension (i.e., $Z_k$ is the union of the strata of dimension $\le k$), yielding the bound $N_X = d_X + 1$.
\end{rem}

\section{Descent}

\subsection{Abstract descent criteria}
\label{ssec:abstract}

Throughout this subsection, $\D$ is a \emph{radditive} preweave, i.e., for every finite family $(X_i)_i$ of algebraic spaces the canonical morphism $\D(\fibcoprod_i X_i) \to \prod_i \D(X_i)$ is invertible (\cite[\S2.3]{weavelisse}).
Given a collection of morphisms $(f_\alpha : Y_\alpha \to X)_\alpha$ of algebraic spaces, we say that a presheaf $F$ is \emph{separated} with respect to $(f_\alpha)_\alpha$ if the functors $f_\alpha^*$ are jointly conservative, and we use the notion of \v{C}ech descent along $(f_\alpha)_\alpha$ from \cite[\S2.3]{weavelisse}.

We record the following variants of \cite[Lem.~2.3.3]{weavelisse} for morphisms admitting $*$- and $!$-direct image.

\begin{lem}\label{lem:abstr*prop}
  Let $\D$ be a radditive left preweave.
  Let $(f_\alpha : Y_\alpha \to X)_\alpha$ be a collection of morphisms such that $\D$ admits $*$-direct image for every morphism obtained by base change and composition from morphisms in $(f_\alpha)_\alpha$.
  Given a subset of indices $\alpha_1,\ldots,\alpha_n$, we write
  $$f_{\alpha_1,\ldots,\alpha_n} : Y_{\alpha_1,\ldots,\alpha_n} := Y_{\alpha_1}\fibprod_X \cdots \fibprod_X Y_{\alpha_n} \to X.$$
  Consider the following conditions:
  \begin{thmlist}
    \item\label{item:abstr*prop/cat}
    The presheaf $\D^*$ satisfies \v{C}ech descent along $(f_\alpha : Y_\alpha \to X)_\alpha$.

    \item\label{item:abstr*prop/sheaf}
    For every $\sF \in \D(X)$ the following is a limit diagram in $\D(X)$:
    \[
      \sF
      \to \prod_\alpha f_{\alpha,*}f_\alpha^*(\sF)
      \rightrightarrows \prod_{\alpha,\beta} f_{\alpha,\beta,*}f_{\alpha,\beta}^*(\sF)
      \rightrightrightarrows
      \cdots
    \]

    \item\label{item:abstr*prop/cons}
    The presheaf $\D^*$ is separated with respect to $(f_\alpha : Y_\alpha \to X)_\alpha$.
  \end{thmlist}
  Then \itemref{item:abstr*prop/cat} $\Rightarrow$ \itemref{item:abstr*prop/sheaf} $\Rightarrow$ \itemref{item:abstr*prop/cons}.
  If $f_\alpha^*$ preserves totalizations of $f_\alpha^*$-split cosimplicial diagrams for each $\alpha$, then all three listed conditions are equivalent.
\end{lem}
\begin{proof}
  We consider the adjoint functors
  \begin{equation}\label{eq:F^*proper}
    F^* : \D(X) \rightleftarrows \Tot\Big(\prod_{\alpha_1,\ldots,\alpha_\bullet} \D(Y_{\alpha_1,\ldots,\alpha_\bullet})\Big) : F_*
  \end{equation}
  as in the proof of \cite[Lem.~2.3.3]{weavelisse}.
  Note that \itemref{item:abstr*prop/cat} holds if and only if $F^*$ is an equivalence, \itemref{item:abstr*prop/sheaf} holds if and only if the unit $\id \to F_* F^*$ is invertible, and \itemref{item:abstr*prop/cons} holds if and only if $F^*$ is conservative.
  It is thus clear that \itemref{item:abstr*prop/cat} $\Rightarrow$ \itemref{item:abstr*prop/sheaf} $\Rightarrow$ \itemref{item:abstr*prop/cons}.

  For the converse direction, we apply (the dual of) \cite[Cor.~4.7.5.3]{LurieHA}.
  Condition (1) holds by the assumption on $f_\alpha^*$, and condition (2) holds because each $f_{\alpha,*}$ commutes with $*$-inverse image (since $\D$ admits $*$-direct image for $f_\alpha$), and conservativity of $\D(X) \to \prod_\alpha \D(Y_\alpha)$ is equivalent to condition~\itemref{item:abstr*prop/cons}.
  The conclusion is that \itemref{item:abstr*prop/cons} implies that $F^*$ is an equivalence, i.e. \itemref{item:abstr*prop/cons} $\Rightarrow$ \itemref{item:abstr*prop/cat}.
\end{proof}

\begin{lem}\label{lem:abstr!prop}
  Let $\D$ be a radditive weave.
  Let $(f_\alpha : Y_\alpha \to X)_\alpha$ be a collection of morphisms.
  Given a subset of indices $\alpha_1,\ldots,\alpha_n$, we write
  $$f_{\alpha_1,\ldots,\alpha_n} : Y_{\alpha_1,\ldots,\alpha_n} := Y_{\alpha_1}\fibprod_X \cdots \fibprod_X Y_{\alpha_n} \to X.$$
  Consider the following conditions:
  \begin{thmlist}
    \item\label{item:abstr!prop/cat}
    The presheaf $\D^!$ satisfies \v{C}ech descent along $(f_\alpha : Y_\alpha \to X)_\alpha$.

    \item\label{item:abstr!prop/sheaf}
    For every $\sF \in \D(X)$ the following is a colimit diagram in $\D(X)$:
    \[
      \cdots
      \rightrightrightarrows \bigoplus_{\alpha,\beta} f_{\alpha,\beta,!}f_{\alpha,\beta}^!(\sF)
      \rightrightarrows \bigoplus_\alpha f_{\alpha,!}f_\alpha^!(\sF)
      \to \sF.
    \]

    \item\label{item:abstr!prop/cons}
    The presheaf $\D^!$ is separated with respect to $(f_\alpha : Y_\alpha \to X)_\alpha$.
  \end{thmlist}
  Then \itemref{item:abstr!prop/cat} $\Rightarrow$ \itemref{item:abstr!prop/sheaf} $\Rightarrow$ \itemref{item:abstr!prop/cons}.
  If the functor $f_\alpha^!$ preserves geometric realizations of $f_\alpha^!$-split simplicial diagrams for each $\alpha$, then all three listed conditions are equivalent.
\end{lem}
\begin{proof}
  The proof is similar, using the functor
  \begin{equation}\label{eq:F^!}
    F^! : \D(X) \to \Tot\Big(\prod_{\alpha_1,\ldots,\alpha_\bullet} \D(Y_{\alpha_1,\ldots,\alpha_\bullet})\Big)
  \end{equation}
  and its left adjoint $F_!$.
  It is clear that \itemref{item:abstr!prop/cat} holds if and only if $F^!$ is an equivalence, \itemref{item:abstr!prop/sheaf} holds if and only if the counit $F_! F^! \to \id$ is invertible, and \itemref{item:abstr!prop/cons} holds if and only if $F^!$ is conservative.
  It is thus clear that \itemref{item:abstr!prop/cat} $\Rightarrow$ \itemref{item:abstr!prop/sheaf} $\Rightarrow$ \itemref{item:abstr!prop/cons}.

  For the converse direction, we again apply \cite[Cor.~4.7.5.3]{LurieHA}.
  Condition (1) holds by the assumption on $f_\alpha^!$, and condition (2) holds because each $f_{\alpha,!}$ commutes with $*$-inverse image (by the base change formula), and conservativity of $\D(X) \to \prod_\alpha \D(Y_\alpha)$ is equivalent to condition~\itemref{item:abstr!prop/cons}.
  The conclusion is that \itemref{item:abstr!prop/cons} implies that $F^!$ is an equivalence, i.e. \itemref{item:abstr!prop/cons} $\Rightarrow$ \itemref{item:abstr!prop/cat}.
\end{proof}

We also record the following simple observation:

\begin{lem}\label{lem:desctarg}
  Let $\D$ be a radditive weave.
  Suppose given a morphism $g : X' \to X$ such that $\D^*$, resp. $\D^!$, satisfies \v{C}ech descent along any base change of $g$.
  Then $\D^*$, resp. $\D^!$, satisfies \v{C}ech descent along any base change of a family $(f_\alpha : Y_\alpha \to X)_\alpha$ if and only if it satisfies \v{C}ech descent along any base change of the induced family
  \[
    (Y_\alpha \fibprod_X X' \to X')_\alpha.
  \]
\end{lem}
\begin{proof}
  The condition is obviously necessary, and the converse direction follows from the fact that limits commute with limits.
\end{proof}

\subsection{Descendability and descent}

\begin{prop}\label{prop:proper descendable implies descent}
  Let $\D$ be a radditive weave.
  Let $f : Y \to X$ be a proper morphism of algebraic spaces.
	If $f$ is descendable in $\D$, then the presheaves $\D^*$ and $\D^!$ satisfy \v{C}ech descent along $f$.
\end{prop}
\begin{proof}
  Let $\sK \in \D(X)$ be an object such that $f^*\sK \simeq 0$ and consider the full subcategory $\sC \sub \D(X)$ spanned by $\sF$ such that $\sK \otimes \sF \simeq 0$.
  Then $\sC$ is a thick subcategory closed under $(-) \otimes \sF'$ for every $\sF' \in \D(X)$.
  By the projection formula for $f_* \simeq f_!$, $\sK \otimes f_* \un_Y \simeq f_* f^* \sK \simeq 0$, i.e., $\sC$ contains $f_* \un_Y$.
  Since $f$ is descendable, it follows that $\un_X \in \sC$, i.e. $\sK \simeq 0$.

  Similarly, if $\sK$ is an object with $f^!\sK \simeq 0$, then
  \[
    \uHom(f_* \un_Y, \sK) \simeq f_*\uHom(\un_Y, f^!\sK) \simeq f_* f^!\sK \simeq 0.
  \]
  Thus $f_*\un_Y$ belongs to the full subcategory $\sC' \sub \D(X)$ spanned by $\sF$ such that $\uHom(\sF, \sK) \simeq 0$.
  This is also a thick subcategory containing $f_*\un_Y$, since for all $\sF'$,
  \[\uHom(\sF \otimes \sF', \sK) \simeq \uHom(\sF', \uHom(\sF, \sK)) \simeq 0. \]
  Applying descendability again, we conclude that $\un_X \in \sC'$, i.e. $\sK \simeq 0$.

  We now show that $\D^*$ satisfies \v{C}ech descent along $f$.
  By \lemref{lem:abstr*prop} and the conservativity of $f^*$, it suffices to show that $f^*$ preserves totalizations of $f^*$-split cosimplicial objects.
  Let $\sK^\bullet$ be such an object and let $\sD \sub \D(X)$ be the full subcategory of those $\sF$ for which the canonical morphism $f^* \Tot(\sK^\bullet \otimes \sF) \to \Tot(f^*(\sK^\bullet \otimes \sF))$ is invertible.
  Then $\sD$ is thick, and contains the essential image of $f_*$: for every $\sG \in \D(Y)$, $\sK^\bullet \otimes f_*\sG \simeq f_*(\sG \otimes f^* \sK^\bullet)$ is split since $f^*\sK^\bullet$ is (and any functor tautologically preserves totalizations of split cosimplicial objects).
  Since the essential image of $f_*$ is closed under $(-) \otimes \sF'$ for every $\sF' \in \D(X)$ (by the projection formula, $f_*(\sG) \otimes \sF' \simeq f_*(\sG \otimes f^*\sF')$), the descendability of $f$ thus implies that $\sD$ contains $\un_X$.
  That is, $f^*$ preserves the totalization of $\sK^\bullet$.

  Dually, for $\D^!$, it suffices using \lemref{lem:abstr!prop} and the conservativity of $f^!$ to check that $f^!$ preserves geometric realizations of $f^!$-split simplicial objects.
  Let $\sK_\bullet$ be such an object and let $\sD' \sub \D(X)$ be the full subcategory of those $\sF$ for which the canonical morphism $\abs{f^!\uHom(\sF, \sK_\bullet)} \to f^! \abs{\uHom(\sF, \sK_\bullet)}$ is invertible.
  Then $\sD'$ is thick, and contains the essential image of $f_*$: by the proper projection formula
  \[
    \uHom(f_* \sG, \sK_\bullet) \simeq f_* \uHom(\sG, f^! \sK_\bullet),
    \qquad \sG \in \D(Y),
  \]
  is split, since $f^! \sK_\bullet$ is.
  By descendability we conclude again that $\sD'$ contains $\un_X$, i.e., $f^!$ preserves the geometric realization of $\sK_\bullet$.
\end{proof}

\subsection{Étale descent}

We say that the weave $\D$ satisfies \emph{étale descent} if $\D^*$ (equivalently $\D^!$) satisfies \v{C}ech descent along any étale surjection.

Recall that a smooth morphism is surjective if and only if it is surjective on $\kappa$-valued points for all separably closed fields $\kappa$.
We say it is \emph{cd-surjective} (where \emph{cd} stands for ``completely decomposed'') if it is surjective on $\kappa$-valued points for all fields $\kappa$, and $\D$ satisfies \emph{Nisnevich descent} if $\D^*$ (equivalently $\D^!$) satisfies \v{C}ech descent along any cd-surjective étale morphism.

\begin{prop}\label{prop:et}
  Let $\D$ be a weave.
  Consider the following conditions:
  \begin{thmlist}
    \item\label{item:et/et} The weave $\D$ satisfies étale descent.
    \item\label{item:et/fet} The presheaf $\D^*$ (equivalently $\D^!$) satisfies \v{C}ech descent along finite étale surjections.
    \item\label{item:et/cons} The presheaf $\D^*$ is separated with respect to finite étale surjections.
  \end{thmlist}
  Then we have \itemref{item:et/et} $\Rightarrow$ \itemref{item:et/fet} $\Rightarrow$ \itemref{item:et/cons}.
  If $\D$ satisfies Nisnevich descent, then also \itemref{item:et/cons} $\Rightarrow$ \itemref{item:et/et}.
\end{prop}
\begin{proof}
  The forward implications are clear.
  Suppose $\D$ satisfies Nisnevich descent.
  Since the étale topology is generated by Nisnevich covers and finite étale covers, we then have \itemref{item:et/et} $\Leftrightarrow$ \itemref{item:et/fet}.
  By \cite[Lem.~2.3.3]{weavelisse}, we also have \itemref{item:et/fet} $\Leftrightarrow$ \itemref{item:et/cons}.
\end{proof}

\begin{cor}\label{cor:fet desc implies et descent}
	Let $\D$ be a radditive weave.
	If every finite étale surjection is descendable, then $\D$ satisfies étale descent.
\end{cor}

\begin{proof}
	Combine \propref{prop:proper descendable implies descent} and \propref{prop:et}.
\end{proof}

\begin{cor}\label{cor:etQ}
  Let $\D$ be an oriented $\bQ$-linear topological weave satisfying continuity and geometric generation.
  Then it satisfies étale descent on qcqs algebraic spaces.
\end{cor}
\begin{proof}
  Combine \corref{cor:fet desc implies et descent} with \corref{cor:fet chi qcqs}.
\end{proof}

\subsection{Proper descent}

\begin{cor}\label{cor:proper descent}
  Let $\D$ be a weave satisfying the localization property.
  If finite étale surjections and finite radicial surjections are descendable, then the presheaves $\D^*$ and $\D^!$ satisfy \v{C}ech descent for proper surjections of finite presentation.
\end{cor}
\begin{proof}
  Let $f : Y \twoheadrightarrow X$ be a proper surjection of finite presentation.
  Since $f$ is proper and descendable (\thmref{thm:descendability}), the claim follows from \propref{prop:proper descendable implies descent}.
\end{proof}

\begin{rem}\label{rem:weave totalization commutes}
  In the situation of \corref{cor:proper descent}, \v{C}ech descent for a proper surjection of finite presentation $f : Y \twoheadrightarrow X$ gives in particular
  \[
    \sF \simeq \Tot( f_{\bullet,*} f_{\bullet}^* \sF ),
    \qquad \text{for all } \sF \in \D(X),
  \]
  where $Y_\bullet$ is the \v{C}ech nerve of $f$ and $f_\bullet : Y_\bullet \to X$ the canonical morphisms.
  The fact that $f$ is \emph{descendable} (\thmref{thm:descendability}) means that this totalization is preserved by any exact functor of presentable stable \inftyCats $F : \D(X) \to \sV$:
  \[
    F \bigl( \Tot( f_{\bullet,*} f_{\bullet}^* \sF ) \bigr)
    \simeq \Tot \bigl( F (f_{\bullet,*} f_{\bullet}^* \sF) \bigr).
  \]
  This follows from \remref{rem:descendable commute with limits} together with the proper projection formula for $f_{\bullet,*}$.
\end{rem}

\subsection{$h$-descent}
\label{ssec:v}

From the previous results we can derive descent for Voevodsky's $h$-topology.
Later, we will also consider the finer $v$-topology of \cite{RydhSubmersions,BhattScholze}, and the even finer arc-topology of \cite{BhattMathew}.

\begin{defn}\leavevmode
  \begin{defnlist}
    \item 
    A morphism of algebraic spaces $f : Y \twoheadrightarrow X$ is an \emph{arc-cover} if it is quasi-compact and, for every morphism $\Spec(V) \to X$ with $V$ a valuation ring of rank $\le 1$, there is an extension $V \hook W$ of rank-$\le 1$ valuation rings and a lift $\Spec(W) \to Y$.
    \item
    A morphism of algebraic spaces $f : Y \twoheadrightarrow X$ is a \emph{$v$-cover} if it is quasi-compact and, for every morphism $\Spec(V) \to X$ with $V$ a valuation ring, there is an extension $V \hook W$ of valuation rings and a lift $\Spec(W) \to Y$.
    \item
    An \emph{$h$-cover} is a $v$-cover of finite presentation.
  \end{defnlist}
\end{defn}

By \cite[Thm.~2.8]{RydhSubmersions}, $v$-covers are the same as \emph{universally subtrusive} quasi-compact morphisms in the sense of \cite[Def.~2.2]{RydhSubmersions}.
Quasi-compact faithfully flat morphisms and proper surjections are $v$-covers.
On noetherian schemes, arc-covers are the same as $v$-covers (see \cite[Thm.~2.8]{RydhSubmersions}), and $h$-covers are $v$-covers of finite type.

\begin{prop}\label{prop:h}
  Let $\D$ be a weave satisfying Nisnevich descent.
  Then $\D^*$ (resp. $\D^!$) satisfies \v{C}ech descent along $h$-covers if and only if it satisfies \v{C}ech descent along proper finitely presented surjections.
\end{prop}
\begin{proof}
  Let $f : Y \twoheadrightarrow X$ be an $h$-cover of algebraic spaces.
  By \cite[Thm.~1.2]{DeshmukhHall}, the algebraic space $X$ admits an \'etale morphism $U \twoheadrightarrow X$ with Nisnevich-local sections, where $U$ is a scheme.
  Choose an affine open cover $U = \bigcup_i U_i$.
  Since $\D^*$ satisfies Nisnevich descent, we may use \lemref{lem:desctarg} to replace $f$ by the base changes $f\fibprod_X U_i$ and thereby assume $X$ is an affine scheme.
  By \cite[Thm.~3.12]{RydhSubmersions}, $f$ admits a refinement $f' : Y' \to X$ which factors as a quasi-compact open covering $p : Y' \twoheadrightarrow X'$ followed by a finitely presented proper surjection $q : X' \twoheadrightarrow X$.
  By assumption, $p$ and $q$ satisfy \v{C}ech descent after any base change, so \cite[Lem.~3.1.2\,(4)]{LiuZheng} now shows that $f'$ satisfies \v{C}ech descent, hence by \cite[Lem.~3.1.2\,(3)]{LiuZheng} $f$ itself satisfies \v{C}ech descent.
\end{proof}

\section{Rational motivic sheaves}

For an algebraic space $S$, we denote by $\SH(S)$ the stable \inftyCat of motivic spectra over $S$.
We denote by $\SH(S)_\Q$ its rationalization, which admits the functorial decomposition (see e.g. \cite[\S 16.2]{CisinskiDegliseBook})
\begin{equation*}
  \SH(S)_\Q \simeq \SH(S)_{\Q,+} \times \SH(S)_{\Q,-}.
\end{equation*}
We will write $\DM(S)_\Q := \SH(S)_{\Q,+}$ and refer to it as the \inftyCat of \emph{rational motivic sheaves} on $S$.
We will write $\Q := \un_{\Q,+}$ for its unit object.

This construction defines a topological weave over arbitrary algebraic spaces, and agrees with other models of rational motivic sheaves defined under more restrictive hypotheses.
For example, it is equivalent to Voevodsky's category when $S$ is geometrically unibranch and excellent (see \cite[Thm.~16.1.4, Thm.~16.2.13]{CisinskiDegliseBook}).

\ssec{Descendability and $h$-descent}

The weave $\SH$ itself satisfies neither étale descent nor topological invariance.
We mention the following positive result.

\begin{prop}
  Let $X$ be a scheme over $\mathbf{F}_p$ and $f : Y\twoheadrightarrow X$ a finite étale surjection of degree $d=p^k$, $k\ge 1$.
  Then $\SH^*[1/p]$ satisfies \v{C}ech descent along $f$.
\end{prop}
\begin{proof}
  We may assume that $X$ is qcqs.
  For every field $\kappa$ over $\mathbf{F}_p$, the class $d_\varepsilon$ is invertible in $\End_{\SH(\kappa)[1/p]}(\un_\kappa)$ by \cite[Lem.~2.2.8]{topinv}, since $d$ is a power of $p$ by assumption.
  As $\SH[1/p]$ satisfies continuity and geometric generation (see e.g. \cite[Exs.~A.4, A.9]{minus}), it follows by \corref{cor:fet chi approx} applied with $B = \Spec(\mathbf{F}_p)$ that the Euler characteristic \eqref{eq:fetcons} is invertible, whence $f$ is descendable of index $\le 1$.
  We conclude by \propref{prop:proper descendable implies descent}.
\end{proof}

After passing to the rationalization $\SH_\Q$, topological invariance holds by \cite{topinv} (cf. \thmref{thm:topinv}).
However, $\SH_\Q$ still does not satisfy étale descent.
The plus part $\SH_{\Q,+} \simeq \DM_\Q$ is an oriented $\Q$-linear topological weave (see e.g. \cite[\S 14, \S 16.2]{CisinskiDegliseBook}) satisfying continuity and geometric generation (see e.g. \cite[Exs.~A.4, A.9]{minus}).
Hence any finite étale surjection of qcqs algebraic spaces has Euler characteristic \eqref{eq:fetcons} invertible in $\DM_\Q$ (\corref{cor:fet chi qcqs}); in particular, it is descendable of index~$\le 1$.
By topological invariance, the same holds for finite radicial surjections (\thmref{thm:topinv}).
Thus \thmref{thm:descendability}, together with \remref{rem:descendability index bound}, yields:

\begin{thm}\label{thm:DMQ descendability}
  Let $f : Y \twoheadrightarrow X$ be a finitely presented surjection of qcqs algebraic spaces.
  Then $f$ is descendable in $\DM_\Q$.
  Moreover, if $X$ is of Krull dimension $d$ then there exists a natural number $N_X$ (depending only on $X$) such that every finitely presented surjection $f : Y \twoheadrightarrow X$ is descendable of index $\le N_X$ in $\DM_\Q$.
\end{thm}

By \corref{cor:fet desc implies et descent}, \corref{cor:proper descent}, and \propref{prop:h} we thus have:

\begin{cor}\label{cor:DMQ descent}
  The presheaves $\DM_\Q^*$ and $\DM_\Q^!$ satisfy $h$-descent on algebraic spaces.
\end{cor}

\corref{cor:DMQ descent} is a folklore result which, using $\infty$-categorical machinery, is not difficult to derive from the work of Cisinski--Déglise in \cite{CisinskiDegliseBook} (compare \cite[Thms.~14.3.4 and 16.2.13]{CisinskiDegliseBook}, over quasi-excellent schemes).
In \secref{sec:fsupp} below, we will require the stronger form of descent provided by \thmref{thm:DMQ descendability}.

\ssec{$v$- and arc-descent}

We next turn to the question of descent for the finer $v$- and arc-topologies.

\begin{quest}\label{quest:DMQ v}
  Does the presheaf $\DM_\Q^*$ satisfy $v$-descent (or arc-descent)?
\end{quest}

We do not see any particular reason to believe the answer is positive.
The following conjecture seems more credible.
We write
\[ \Cmot(X; \sF) := \Gamma(X; \sF), \qquad \Hmot^s(X; \sF) := \pi_{-s} \Cmot(X; \sF) \]
for $\sF \in \DM(X)_\Q$ and $s \in \Z$.

\begin{conj}\label{conj:v}
  For every integer $r\in\Z$, the presheaf $\Cmot(-; \Q(r))$ satisfies arc-descent.
\end{conj}

We will take a rather heavy-handed approach to this problem, assuming the following generalized Beilinson--Soul\'e type vanishing conditions for valuation rings, for each $r \in \Z$:

\begin{enumerate}
  \item[$(\mathrm{BS}_{r})$]\label{cond:BSr} There exists an integer $N_r$ such that for every finite rank henselian valuation ring $V$, $\Hmot^m (V; \Q(r)) = 0$ for all $m<-N_r$.
\end{enumerate}

\begin{thm}\label{thm:BS arc}
  Let $r \in \Z$.
  If \BSr holds, then the presheaf $\Cmot(-; \Q(r))$ satisfies $v$- and arc-hyperdescent on algebraic spaces.
\end{thm}

Unconditionally, we deduce:

\begin{cor}\label{cor:Cmot arc low weights}
  For $r \le 1$, the presheaf $\Cmot(-; \Q(r))$ satisfies $v$- and arc-hyperdescent.
\end{cor}

\begin{proof}
  By \thmref{thm:BS arc} it suffices to verify \BSr for $r\le 1$.
  In fact, for every qcqs scheme $U$ we have:
  \begin{enumerate}
    \item 
    For $r<0$, $\Hmot^s(U; \Q(r)) \simeq 0$ for all $s\in\Z$.
    In the stated generality, this may be seen as follows: the Tate twist $\Q(r)$ is identified with $\KGL_\Q^{(r)}[-2r]$, the shifted $r$-th Adams eigenspace of $\KGL_\Q$ (see \cite{Riou}, \cite[Lem.~14.1.4]{CisinskiDegliseBook}).
    By \cite[Cor.~4.52(2), Rem.~4.22]{BachmannElmantoMorrow}, this is the rationalization of the $r$-th slice $s^r \KGL$, and $\Gamma(U, s^r \KGL) := \Maps_{\SH(U)}(\un_U, s^r \KGL) \simeq 0$ holds by construction of the slice filtration.

    \item
    For $r = 0$, $\Cmot(U; \Q) \simeq \Gamma_{\cdh}(U, \Q)$ is the cdh hypercohomology of the constant sheaf $\Q$ (see e.g. \cite[Thm.~1.1(6)]{BachmannElmantoMorrow}), hence $\Hmot^{s}(U; \Q) = 0$ for $s < 0$.
    
    \item
    For $r = 1$, $\Cmot(-; \Q(1)) \simeq \Gamma_{\cdh}(-; \bG_m \otimes \Q)[-1]$ (see e.g. \cite[Thm.~1.1(6)]{BachmannElmantoMorrow}), hence $\Hmot^{s}(U; \Q(1)) \simeq \mathrm{H}^{s-1}_{\cdh}(U; \bG_m)_{\Q} \simeq 0$ for $s \le 0$ since $\bG_m$ is discrete.\qedhere
  \end{enumerate}
\end{proof}

The proof of \thmref{thm:BS arc} will make use of the following lemmas.

\begin{lem}\label{lem:BS propagate}
  Let $r \in \Z$.
  The condition \BSr is equivalent to the existence of an integer $N_r$ such that $\Cmot(X; \Q(r))$ is $N_r$-coconnective for every qcqs scheme $X$.
\end{lem}
\begin{proof}
  Set $F := \Cmot(-; \Q(r))$ and let $N_r$ be an integer such that $F(V)$ is $N_r$-coconnective for every henselian valuation ring $V$ of finite rank.
  Since $\DM_\Q$ satisfies continuity and cdh descent, the same holds for $F$.
  Throughout the following, we regard $F$ as a presheaf on the site of qcqs schemes, or equivalently by continuity, on the site of schemes finitely presented over $\Spec(\bZ)$ (see \cite[Prop.~2.1.6]{EHIKMathAnn}).

  Let $K := \tau_{\ge N_r+1}(F)$ denote the objectwise $(N_r+1)$-connective cover of $F$, whose cdh sheafification $K_{\mrm{cdh}} := \tau_{\ge N_r+1}^{\mrm{cdh}}(F)$ is the connective cover in cdh sheaves of spectra.
  Since $\tau_{\ge N_r+1}$ commutes with filtered colimits of spectra, $K$ still satisfies continuity, and hence so does $K_{\mrm{cdh}}$ by \cite[Prop.~2.15]{BachmannElmantoMorrow}.

  By noetherian approximation, any qcqs scheme $X$ may be written as the limit of a cofiltered system of schemes of finite type over $\Z$ with affine transition maps.
  Fix an index $\alpha$ and a finite rank henselian valuation ring $V$ over $X_\alpha$.
  Since henselian valuation rings are the points of the cdh topos (see \cite[\S 2.1]{EHIKMathAnn}), we have $K_{\mrm{cdh}}(V) \simeq K(V) \simeq \tau_{\ge N_r+1} F(V)$, which vanishes by \BSr.
  Since $X_\alpha$ is of finite valuative dimension (see \cite[Prop.~2.3.2 (9)]{EHIKMathAnn}), we may now apply \cite[Cor.~2.4.19]{EHIKMathAnn} to deduce that $K_{\mrm{cdh}}|_{X_\alpha}$ vanishes as a cdh sheaf on schemes over $X_\alpha$.
  In particular, $K_{\mrm{cdh}}(X_\alpha) \simeq 0$.
  Since this holds for all $\alpha$, the continuity of $K_{\mrm{cdh}}$ now implies that $K_{\mrm{cdh}}(X) \simeq 0$ for every qcqs scheme $X$.

  Consider now the exact triangle $\tau_{\ge N_r+1}^{\mrm{cdh}} (F) \to F \to \tau_{\le N_r}^{\mrm{cdh}} (F)$ of cdh sheaves of spectra.
  By the previous paragraph, we have $F(X) \simeq \tau_{\le N_r}^{\mrm{cdh}} F(X)$, which is $N_r$-truncated.
\end{proof}

\begin{lem}\label{lem:colim Tot}
  Let $\sC$ be a presentable stable \inftyCat equipped with a $t$-structure $(\sC_{\ge 0}, \sC_{\le 0})$ which is compatible with filtered colimits and right-separated.
  Let $(M_\alpha^\bullet)_\alpha$ be a filtered diagram of cosimplicial objects of $\sC$ and $N$ an integer such that $M_\alpha^n \in \sC_{\le N}$ for all $\alpha$ and $[n] \in \bDelta$.
  Then the canonical morphism
  \begin{equation}\label{eq:colim Tot}
    \colim_\alpha \Tot(M_\alpha^\bullet) \to \Tot\bigl(\colim_\alpha M_\alpha^\bullet\bigr)
  \end{equation}
  is invertible.
\end{lem}
\begin{proof}
  For any cosimplicial object $M^\bullet$, the fibre of $\Tot_{\le m}(M^\bullet) \to \Tot_{\le m-1}(M^\bullet)$ is $M'[-m]$ for $M'$ a direct summand of $M^m$ (by the dual of \cite[Rem.~1.2.4.3]{LurieHA}).
  If $M^n \in \sC_{\le N}$ for all $n \in \bDelta$, it follows that the fibre of $\Tot_{\le m}(M^\bullet) \to \Tot_{\le m'}(M^\bullet)$ lies in $\sC_{\le N-m'-1}$ for all $m > m' \ge 0$.
  Passing to the limit over $m$, the same holds for the fibre of $\Tot(M^\bullet) \to \Tot_{\le m'}(M^\bullet)$ (since coconnectivity is stable under limits).
  In particular $\Tot(M^\bullet) \in \sC_{\le N}$ and $\pi_k \Tot(M^\bullet) \to \pi_k \Tot_{\le m'}(M^\bullet)$ is invertible whenever $m' \ge N - k + 1$.

  The source and target of \eqref{eq:colim Tot} both lie in $\sC_{\le N}$, since $\sC_{\le N}$ is closed under limits and filtered colimits.
  Fix $k$ and set $m' = N - k + 1$.
  Since $\Tot_{\le m'}$ commutes with filtered colimits (as a finite limit), and $\pi_k$ commutes with filtered colimits by assumption, we have
  \[
    \pi_k \bigl( \colim_\alpha \Tot(M_\alpha^\bullet) \bigr)
    \simeq \colim_\alpha \pi_k \Tot_{\le m'}(M_\alpha^\bullet)
    \simeq \pi_k \Tot_{\le m'}\bigl(\colim_\alpha M_\alpha^\bullet\bigr)
    \simeq \pi_k \Tot\bigl(\colim_\alpha M_\alpha^\bullet\bigr),
  \]
  compatibly with the canonical morphism, where the last isomorphism uses $\colim_\alpha M_\alpha^p \in \sC_{\le N}$.
  The fibre of \eqref{eq:colim Tot} therefore lies in $\sC_{\le N}$ and has vanishing homotopy objects; in particular it lies in $\bigcap_n \sC_{\le -n}$, hence is zero by the right-separatedness assumption.
\end{proof}

\begin{proof}[Proof of \thmref{thm:BS arc}]
  Condition~\BSr implies that the presheaf $\Cmot(-; \Q(r))$ takes values in $N_r$-coconnective spectra (\lemref{lem:BS propagate}).
  Hence by \cite[Lem.~3.1.7 (2), Ex.~3.1.6]{EHIKMathAnn}\footnote{%
    Following \cite[Def.~4.16]{BhattMathew}, we define the $v$- and arc-toposes of sheaves on the (non-essentially-small) site of qcqs schemes by fixing an uncountable strong limit cardinal $\kappa$ and restricting to $\kappa$-bounded qcqs schemes as in \emph{loc. cit}.
    Since $\Cmot(-;\Q(r))$ is continuous, the statement is independent of $\kappa$.
  }, it is a $v$- or arc-hypersheaf if and only if it is a $v$- or arc-sheaf.
  
  Let $f : Y \twoheadrightarrow X$ be a $v$-cover of algebraic spaces.
  The claim is étale-local on $X$ by \corref{cor:etQ} and smooth base change, so we may assume $X$ affine.
  Since $f$ is quasi-compact, $Y$ is quasi-compact and hence admits an étale surjection $p : Y' \twoheadrightarrow Y$ with $Y'$ affine.
  By \v{C}ech descent for $p$ and its base changes, \cite[Lem.~3.1.2\,(3),(4)]{LiuZheng} reduces \v{C}ech descent along $f$ to \v{C}ech descent along $g := f \circ p$.
  Thus we may assume $X$ and $Y$ are affine.
  By \cite[Lem.~2.12]{BhattScholze}, $f$ is then the limit of a cofiltered diagram of $h$-covers $(f_\alpha : Y_\alpha \to X)_\alpha$ with affine transition maps.
  By continuity (\cite[Rem.~2.76]{weaves}; see \cite[Exs.~A.4, A.9]{minus}) the canonical maps
  \[\colim_\alpha \Cmot(Y_{\alpha,n}; \Q) \to \Cmot(Y_n; \Q)\]
  are invertible for each $[n]\in\bDelta$.
  By \v{C}ech descent along each $f_\alpha$ (\corref{cor:DMQ descent}) it will suffice to show that the filtered colimit on the left-hand side commutes with the totalization over $[n]\in\bDelta$.
  This follows from \lemref{lem:colim Tot} in view of the fact that $\Cmot(-; \Q(r))$ takes values in $N_r$-coconnective spectra.

  For arc-descent, we apply the criterion of \cite[Thm.~4.1]{BhattMathew} to $\Cmot(-; \Q(r))$ regarded as a presheaf with values in $N_r$-coconnective spectra; this \inftyCat is compactly generated by cotruncated objects (see \cite[Ex.~3.1.6]{EHIKMathAnn}).
  Since $\Cmot(-; \Q(r))$ satisfies continuity and $v$-descent, it remains to check that for every absolutely integrally closed valuation ring $V$ and every prime ideal $\mathfrak{p} \sub V$, the square
  \[\begin{tikzcd}
  \Cmot(V; \Q(r)) \ar[r] \ar[d] & \Cmot(V/\mathfrak{p}; \Q(r)) \ar[d] \\
  \Cmot(V_{\mathfrak{p}}; \Q(r)) \ar[r] & \Cmot(\kappa(\mathfrak{p}); \Q(r))
  \end{tikzcd}\]
  is cartesian.
  By \cite[Prop.~2.8]{BhattMathew}, this is a special case of Milnor excision for $\Cmot(-; \Q(r))$ (see \cite[Cor.~1.2]{EHIKCrelle}).
\end{proof}

\section{Étale motivic spectra}

Let $X$ be an algebraic space.
We denote by $\SH_{\et}(X)$ the \inftyCat of étale motivic spectra over $X$; that is, the full subcategory of $\SH(X)$ spanned by motivic spectra satisfying étale hyperdescent.
The assignment $X \mapsto \SH_\et(X)$ forms a topological weave satisfying étale hyperdescent (e.g. by \cite[Prop.~2.3.2]{weavelisse}).
We write $\sS$ for the category of noetherian algebraic spaces of finite Krull dimension, $\sS_0$ for the full subcategory of those whose residue fields have uniformly bounded virtual étale cohomological dimension, and $\sS_1 \sub \sS_0$ for the further subcategory of those whose residue fields have uniformly bounded étale cohomological dimension.

\ssec{Descendability and $h$-descent}

Our main result is as follows:

\begin{thm}\label{thm:SHet descendable}
  For $X \in \sS_1$, every finite type surjection $f : Y \twoheadrightarrow X$ is descendable in $\SH_\et$.
  For $X \in \sS_0$, every finite type surjection $f : Y \twoheadrightarrow X$ is descendable in $\SH_\et[1/2]$.
\end{thm}

By \corref{cor:proper descent} and \propref{prop:h} this implies:

\begin{cor}\label{cor:SHet descent}
  On the category $\sS_1$, the presheaves $\SH^*_{\et}$ and $\SH^!_{\et}$ satisfy $h$-descent.
  Similarly for the presheaves $\SH^*_{\et}[1/2]$ and $\SH^!_{\et}[1/2]$ on $\sS_0$.
\end{cor}

\subsection{Proof of \thmref{thm:SHet descendable}}

We first observe that any $X \in \sS_0$ is of uniformly bounded $p$-étale cohomological dimension $\le c$ for all odd $p$: by \cite[Cor.~2.13]{BachmannRigidity} it suffices to check that $\sup \mathrm{cd}_p k(x) < \infty$, over $x \in X$ and odd primes $p$; but $\mathrm{cd}_p k(x) \le \on{vcd} k(x)$, so this follows by the assumption.
If moreover $X \in \sS_1$, then the same holds at all primes $p$, since $\mathrm{cd}_p k(x) \le \mathrm{cd}\, k(x)$ is then uniformly bounded by definition.

\begin{lem}\label{lem:SHet no minus}
  On the category $\sS$, there is a canonical equivalence $\SH_\et[1/2]_- \simeq 0$.
\end{lem}
\begin{proof}
  There is a functorial splitting $\SH_\et[1/2] \simeq \SH_\et[1/2]_+ \times \SH_\et[1/2]_-$, where $\langle -1 \rangle \simeq 1$ on the plus part.
  The minus part is equivalently identified with the $\eta$-periodization $\SH_\et[1/2][\eta^{-1}]$ for $\eta$ the algebraic Hopf map (see e.g. \cite[Rem.~3.6]{ALPWitt}, \cite[(1.0.a)]{minus}); this vanishes by \cite[Thm.~A]{MattisTubach}, i.e. $\SH_\et[1/2] \simeq \SH_\et[1/2]_+$.
\end{proof}

\begin{lem}\label{lem:SHet stalks}
  The weave $\SH_{\et}$ satisfies conservativity of stalks for every $X \in \sS_0$.
\end{lem}
\begin{proof}
  By \cite[Cor.~5.12]{BachmannRigidity}, it will suffice to show that $X$ admits an \'etale cover by schemes over which every finite type scheme is of uniformly bounded étale cohomological dimension.
  Consider the étale cover $\{X_i, X_\omega\}$ of $X$, where
  \[ X_i := X \otimes_\bZ \bZ[\tfrac{1}{2}, x]/(x^2+1), \qquad X_\omega := X \otimes_\bZ \bZ[\tfrac{1}{3}, x]/(x^2+x+1), \]
  cf. \cite[Ex.~2.14]{BachmannRigidity}.
  Every residue field $K$ of $X_i$ (resp. $X_\omega$) is generated over a residue field $\kappa$ of $X$ by a primitive fourth (resp. third) root of unity; in particular, $-1$ (resp. $-3$) is a square in $K$, so $K$ is not formally real, and its absolute Galois group is therefore torsion-free by the Artin--Schreier theorem.
  By \cite[I.\S3.3, Props.~14 and 14$'$]{SerreGC}, it then follows that
  \[ \mathrm{cd}_p(K) = \mathrm{cd}_p(K(\sqrt{-1})) \le \on{vcd}(K) \le \on{vcd}(\kappa) \]
  for every prime $p$.
  Thus $X_i$ and $X_\omega$ have residue fields of uniformly bounded étale cohomological dimension, and the claim now follows by \cite[Cor.~2.13]{BachmannRigidity}.
\end{proof}

\begin{prop}\label{prop:SHet topinv}
  On the category $\sS_0$, the weave $\SH^*_{\et}$ satisfies topological invariance.
\end{prop}
\begin{proof}
  For every prime $\ell$, $\ell$ is invertible in $\End_{\SH_{\et}(S)}(\un_S)$ where $S = \Spec(\bF_\ell)$ by \cite[Thm.~A.1]{BachmannHoyoisEtale}.
  It follows that for every field $\kappa$ of characteristic $p>0$, $p$ is invertible in $\End_{\SH_{\et}(S)}(\un_S)$ where $S = \Spec(\kappa)$.
  By \thmref{thm:topinv} it follows that $f^*$ is conservative for any finite radicial surjection $f : Y \twoheadrightarrow X$ with $X$ the spectrum of a field.
  The general case then follows by \lemref{lem:frad} and conservativity of stalks (\lemref{lem:SHet stalks}).
\end{proof}

In order to show descendability for finitely presented surjections in $\SH_\et$, the only obstruction after \thmref{thm:descendability}, \lemref{lem:fet desc torsor}, and \propref{prop:SHet topinv}, is the case of torsors under finite groups.

\begin{prop}\label{prop:SHet fet desc}
  Let $X \in \sS_0$ with uniformly bounded $p$-étale cohomological dimension $\le c$ for every odd prime $p$.
  Let $f : Y \to X$ be a $G$-torsor where $G$ is a finite group of order $d$.
  Then $f$ is descendable of index $\le 2(c+1)$ in $\SH_\et(X)[1/2]$.
  If $X$ is moreover of uniformly bounded $2$-étale cohomological dimension $\le c$, then $f$ is descendable of index $\le 2(c+1)$ in $\SH_\et(X)$.
\end{prop}

We treat the two cases in parallel; all objects below are formed in the weave $\SH_\et[1/2]$, resp. $\SH_\et$ in the second case.
To simplify notation we write $\un := \un_X$ and $A := f_*(\un_Y)$.
By \lemref{lem:descendable local abstract}, we may divide the proof of \propref{prop:SHet fet desc} into the following two statements:

\begin{enumerate}[
    before=\itshape,
    label={\textup{\hyperref[prop:SHet fet desc]{(\ref*{prop:SHet fet desc}.\arabic*)}}},
    ref={(\ref*{prop:SHet fet desc}.\arabic*)}]
  \item\label{item:SHet fet desc/P_d}
  Let $\un^{\wedge}_d$ be the $d$-completion of $\un$ and consider the $E_\infty$-algebra $P_d := \un[1/d] \times \un^{\wedge}_d$.
  The unit map $\un \to P_d$ is descendable of index $\le 2$.

  \item\label{item:SHet fet desc/A}
  \xdef\SHetfetdescAnum{\arabic{enumi}}
  The unit $\un \to A$ becomes descendable of index $\le c+1$ after extension of scalars along $\un \to P_d$.
\end{enumerate}

\subsection{Proof of \itemref{item:SHet fet desc/P_d}}

Recall the cartesian square (see e.g. \cite[Lem.~3.19]{BachmannHopkinsEta})
\[\begin{tikzcd}
  \un \ar{r}\ar{d}
  & \un^\wedge_d \ar{d}
  \\
  \un[1/d] \ar{r}
  & \un^{\wedge}_d[1/d].
\end{tikzcd}\]
This yields the exact triangle $\un \to P_d \to Q_d$, where $Q_d := \un^{\wedge}_d[1/d]$.
The claim is that for $I_d \simeq Q_d[-1]$ the fibre of the unit map $\un \to P_d$, the map $\varepsilon^2 : I_d^{\otimes 2} \to \un$ is null-homotopic.

Now $Q_d$ is a module over $\un^{\wedge}_d$, hence a direct summand of $\un^{\wedge}_d \otimes Q_d$ via the identity
\[
  \id \simeq \bigl[ Q_d \simeq \un\otimes Q_d \xrightarrow{\mrm{unit} \otimes \id} \un^{\wedge}_d \otimes Q_d \xrightarrow{\mrm{act}} Q_d \bigr].
\]
Tensoring the splitting $P_d = \un[1/d]\oplus\un^{\wedge}_d$ with $Q_d$ exhibits $\un^{\wedge}_d \otimes Q_d$, and hence $Q_d$, as a direct summand of the free $P_d$-module $P_d \otimes Q_d$.
Shifting, $I_d \simeq Q_d[-1]$ is a direct summand of the $P_d$-module $P_d \otimes I_d$.

To prove that $\varepsilon^2 : I_d^{\otimes 2} \to \un$ is null-homotopic we make use of the following observation:

\begin{enumerate}
	\item[$(\ast)$] For every $P_d$-module $N$, the map $\varepsilon \otimes \id_N : I_d \otimes N \to \un \otimes N \simeq N$ is null-homotopic.
\end{enumerate}
Indeed, consider the composite
\[
  I_d \otimes N \xrightarrow{\varepsilon \otimes \id_N}
  \un \otimes N \xrightarrow{\mrm{unit} \otimes \id_N}
  P_d \otimes N \xrightarrow{\mrm{act}} N.
\]
The composite of the two rightmost arrows is the identity of $N$, while the composite of the two leftmost is null-homotopic (via the null-homotopy of $I_d \to \un \to P_d$).
It follows that the entire composite is $\varepsilon \otimes \id_N$ and that it is null-homotopic.

We apply $(\ast)$ to $N := P_d \otimes I_d$.
Since $I_d$ is a direct summand of $N$, the map $\varepsilon \otimes \id_I : I_d^{\otimes 2} \to I_d$ is a direct summand of $\varepsilon \otimes \id_N$, hence also null-homotopic.
Hence the composite
\[
  \varepsilon^{\otimes 2} : I_d^{\otimes 2} \xrightarrow{\varepsilon \otimes \id_I} I_d \xrightarrow{\varepsilon} \un
\]
is also null-homotopic, as claimed.

\subsection{Proof of \itemref{item:SHet fet desc/A}}

The claim is that
\[A \otimes P_d \simeq f_*(\un_Y) \otimes (\un[1/d]\times \un^{\wedge}_d) \simeq f_*\bigl(\un_Y[1/d]\bigr) \times f_*\bigl((\un_Y)^{\wedge}_d\bigr)\]
is descendable of index $\le c+1$ over $P_d$.
This amounts to the claim that $f$ is descendable of index $\le c+1$ in the weaves $\SH_{\et}[1/2, 1/d]$ and $\SH_{\et}[1/2]^{\wedge}_d$ (resp. $\SH_{\et}[1/d]$ and $\SH_{\et}{}^{\wedge}_d$, in the second case).
For $\SH_{\et}[1/2, 1/d]$ (resp. $\SH_{\et}[1/d]$) the claim follows from \corref{cor:fet chi stalkwise} (where conservativity of stalks for $X$ holds by \lemref{lem:SHet stalks}) and the following lemma, which was suggested to us by Bachmann.

\begin{lem}\label{lem:SHet deps}
  For every field $\kappa$ and every natural number $d$, the class $d_\varepsilon$ is invertible in $\End_{\SH_{\et}(\kappa)[1/d]}(\un_\kappa)$ (where the notation is as in \lemref{lem:fet chi field}).
\end{lem}
\begin{proof}
  Recall that $d_\varepsilon$ is by definition the sum of $\langle (-1)^{i-1} \rangle$ over $1\le i\le d$, and that the class $\langle a \rangle \in \End_{\SH(\kappa)}(\un_\kappa)$ of a unit $a \in \kappa^\times$ depends only on its square class: under the identification $\End_{\SH(\kappa)}(\un_\kappa) \simeq \on{GW}(\kappa)$ of Morel (cf. \cite[Thm.~10.12]{BachmannHoyois}), it is the quadratic form $\langle a \rangle$.
  In particular, if $-1$ is a square in $\kappa$ (e.g. if $\kappa$ is of characteristic $2$), then $\langle -1 \rangle \simeq \langle 1 \rangle \simeq 1$, so $d_\varepsilon \simeq d$ is invertible in $\End_{\SH_{\et}(\kappa)[1/d]}(\un_\kappa)$.

  In general, we may assume $\kappa$ is not of characteristic $2$, so that $\kappa' := \kappa[x]/(x^2+1)$ is a finite étale $\kappa$-algebra of degree $2$, in every residue field of which $-1$ is a square.
  Since $\SH_\et$ satisfies étale descent, the pullback functor $p^*$ along $p : \Spec(\kappa') \twoheadrightarrow \Spec(\kappa)$ is conservative, and remains so after restricting to $[1/d]$-local objects.
  The class $d_\varepsilon$ is stable under base change, so by the previous paragraph $d_\varepsilon$ becomes a unit in $\SH_\et(\kappa')[1/d]$, hence by conservativity is invertible already in $\SH_\et(\kappa)[1/d]$.
\end{proof}

For $\SH_{\et}[1/2]^{\wedge}_d$, which is by definition the product of $\SH_{\et}[1/2]^{\wedge}_p$ over all odd primes $p$ dividing $d$ (resp. $\SH_{\et}{}^{\wedge}_d$, the product of $\SH_{\et}{}^{\wedge}_p$ over all primes $p$ dividing $d$), we argue as follows, uniformly in $p$.
By \cite[Thm.~A.1]{BachmannHoyoisEtale}, $\SH_{\et}(Z_p)^\wedge_p \simeq 0$ for $Z_p := X \otimes_\bZ \bF_p$, the characteristic $p$ part of $X$.
Hence by \lemref{lem:descendability and localization} (since $p$-completion is exact, $\SH_{\et}(-)^{\wedge}_p$ still satisfies the localization property), we may replace $X$ by $X[1/p]$ and thereby assume $p$ is invertible on $X$.
Then Bachmann's rigidity equivalence
\[ \SH_{\et}(X)^\wedge_p \simeq \widehat{\Shv}(X_{\et})^\wedge_p \]
identifies $\SH_{\et}(X)^\wedge_p$ with the $p$-completion of the $\infty$-category of hypercomplete sheaves of spectra on the small étale site of $X$ (see \cite[Thm.~3.1]{BachmannRemarks}, which requires only that $p$ be invertible on $X$; cf. \cite[Thm.~6.6]{BachmannRigidity}).
This reduces us to the following statement (where $f : Y \to X$ is still as in \propref{prop:SHet fet desc}):

\begin{enumerate}[
    wide=0pt, labelsep=0.5em, before=\itshape,
    label={\textup{\hyperref[prop:SHet fet desc]{(\ref*{prop:SHet fet desc}.\SHetfetdescAnum.\alph*)}}},
    ref={(\ref*{prop:SHet fet desc}.\SHetfetdescAnum.\alph*)}]
  \item\label{item:SHet fet desc/P_d/Shv}
  The morphism $f$ is descendable of index $\le c+1$ in $\widehat{\Shv}(X_{\et})^\wedge_p$.
\end{enumerate}

We write $\sC := \widehat{\Shv}(X_{\et})^\wedge_p$ and denote its unit object by $\un$.
Assertion \itemref{item:SHet fet desc/P_d/Shv} is precisely that the map $\un \to \Tot_{\le c}(A^\bullet)$ admits a retraction in $\sC$, where $A^\bullet$ is the \v{C}ech nerve, with $A^n = A^{\otimes n+1}$.
Writing $f_\bullet : Y_\bullet \to X$ for the \v{C}ech nerve of $f : Y \twoheadrightarrow X$, and $\Sigma^\infty_+ Y_n$ for the object of $\sC$ represented by $Y_n$, we have
\[
  \uHom(\Sigma^\infty_+ Y_\bullet, \un) \simeq f_{\bullet,*}f_\bullet^*(\un) \simeq A^\bullet
\]
using $f_{n,*}f_n^* \simeq f_{n,!}f_n^!$ (since each $f_n$ is finite \'etale) and the projection formula.
Writing $\abs{\on{sk}_i(Y_\bullet)}$ for the geometric realization of the $i$-skeleton, it will thus suffice to show that for $F_i$ the fibres
\[ F_i \to \abs{\on{sk}_i \Sigma^\infty_+ Y_\bullet} \to \un, \]
the boundary map $\partial : \un[-1] \to F_c$ is null-homotopic.
We claim more generally that for every $i \ge -1$ the map $F_i \to F_{i+c+1}$ is null-homotopic (for $i=-1$, $F_{-1} \simeq \un[-1]$ and this is the boundary $\partial$).
By a standard Postnikov filtration argument, it will suffice to show the following:

\begin{enumerate}[label={(\roman*)}]
  \item\label{item:Fi Post} $F_{i+c+1}$ is Postnikov-complete.
	\item\label{item:Fi cd} $F_i$ is of cohomological dimension $\le c+i$ for every $i \ge -1$.
	\item\label{item:Fi conn} $F_{i+c+1}$ is $(i+c+1)$-connective.
\end{enumerate}

Here we say that an object $\sF \in \sC$ is of \emph{cohomological dimension $\le N$} if
\[
  H^s(\sF; \sG) := \pi_0 \Maps_\sC(F, \sG[s]) \simeq 0
\]
for all $\sG \in \sC^\heartsuit$ and all $s>N$.
This condition is closed under extensions and direct summands.

\begin{lem}\label{lem:coh dim heart}
  For every finitely presented \'etale $U \to X$, the object $\Sigma^\infty_+ U \in \sC$ is of cohomological dimension $\le c$.
\end{lem}
\begin{proof}
  Since $X$ is of uniformly bounded $p$-étale cohomological dimension $\le c$ (by the hypothesis of \propref{prop:SHet fet desc}), we have
  \[
    H^s(\Sigma^\infty_+ U; \sG) \simeq
    H^s_{\et}(U; \sG) \simeq
    0
  \]
  for all $s>c$ and all \emph{$p$-torsion} $\sG \in \sC^\heartsuit$.
  For $\sG \in \sC^\heartsuit$ arbitrary, denote by $\sG\modmod p^k$ the cofibre
  \[
    \sG \xrightarrow{p^k} \sG \to \sG\modmod p^k.
  \]
  Note that $\pi_0(\sG\modmod p^k) \simeq \sG/p^k \sG$ and $\pi_1(\sG\modmod p^k)$ is the $p^k$-torsion subsheaf $\sG_{\mrm{tors}} \sub \sG$, so $\sG\modmod p^k$ is $p^k$-torsion.
  From the exact triangle $\sG_{\mrm{tors}}[1] \to \sG\modmod p^k \to \sG/p^k \sG$ where each term is $p^k$-torsion, we read off: $H^s_\et(U; \sG\modmod p^k) \simeq 0$ for $s>c$ and $H^c_\et(U; \sG\modmod p^k) \simeq H^c_\et(U; \sG/p^k \sG)$, the latter using the vanishing of $H^{>c}_\et(U; \sG_{\mrm{tors}})$.
  Note that this second isomorphism is compatible with transition maps, i.e. an isomorphism of pro-systems.

  As $\sG$ is derived $p$-complete, i.e. $\sG \simeq \lim_k \sG\modmod p^k$, we have
  \[ \Maps_{\sC}(\Sigma^\infty_+ U, \sG) \simeq \lim_k \Maps_{\sC}(\Sigma^\infty_+ U, \sG\modmod p^k) \simeq \lim_k R\Gamma_{\et}(U; \sG\modmod p^k) \]
  and its homotopy groups sit in the Milnor sequence
  \[ 0 \to \sideset{}{^1}\lim_k H_\et^{s-1}(U; \sG\modmod p^k) \to H^s(\Sigma^\infty_+ U; \sG) \to \lim_k H^s_\et(U; \sG\modmod p^k) \to 0. \]
  We know that the right-hand term vanishes for $s>c$ and the left-hand term vanishes for $s \ge c+2$.
  In particular, it remains to prove that $H^{c+1}(\Sigma^\infty_+U; \sG) \simeq \lim\nolimits^1_k H_\et^c(U; \sG/p^k \sG)$ vanishes.

  Using the short exact sequence
  \[0 \to p^k \sG/p^{k+1}\sG \hook \sG/p^{k+1} \sG \twoheadrightarrow \sG/p^k \sG \to 0,\]
  where the kernel is $p$-torsion, we obtain the long exact sequence
  \[
    \cdots \to H_\et^c(U; \sG/p^{k+1} \sG) \xrightarrow{\phi_k} H_\et^c(U; \sG/p^k \sG) \xrightarrow{\partial} H_\et^{c+1}(U; p^k \sG/p^{k+1} \sG) \to \cdots,
  \]
  where $H_\et^{c+1}(U; p^k \sG/p^{k+1} \sG) \simeq 0$.
  Thus the transition maps $\phi_k$ are surjective, the pro-system $\{ H_\et^c(U; \sG/p^k \sG) \}_k$ satisfies the Mittag--Leffler condition, and $\lim\nolimits^1_k H_\et^c(U; \sG/p^k \sG) \simeq 0$.
  In particular, $H^s(\Sigma^\infty_+ U; \sG) \simeq 0$ for all $s>c$.
\end{proof}

We now use \lemref{lem:coh dim heart} to conclude the proof as follows.
First, $\on{sk}_0(\Sigma^\infty_+ Y_\bullet) \simeq \Sigma^\infty_+ Y$ is of cohomological dimension $\le c$.
By construction there are exact triangles (see \cite[Thm.~1.2.4.1, Rem.~1.2.4.3]{LurieHA})
\[ \abs{\on{sk}_{i-1} \Sigma^\infty_+ Y_\bullet} \to \abs{\on{sk}_{i} \Sigma^\infty_+ Y_\bullet} \to C_i[i], \]
where each $C_i$ is a direct summand of $\Sigma^\infty_+ Y_i$.
Thus by \lemref{lem:coh dim heart} again, $C_i[i]$ is of cohomological dimension $\le c+i$.
By induction, we conclude that $\abs{\on{sk}_i(\Sigma^\infty_+ Y_\bullet)}$ is of cohomological dimension $\le c+i$, for every $i \ge 0$.

\textit{Proof of \itemref{item:Fi cd}.}
For $i=-1$, $F_{-1} \simeq \un[-1] \simeq \Sigma^\infty_+ X[-1]$ is of cohomological dimension $\le c-1$.
Thus in the exact triangle $\un[-1] \to F_i \to \abs{\on{sk}_i \Sigma^\infty_+ Y_\bullet}$, the outer terms are both of cohomological dimension $\le c+i$, hence so is $F_i$.

\textit{Proof of \itemref{item:Fi conn}.}
The cofibre $F_i[1]$ of $\abs{\on{sk}_i \Sigma^\infty_+ Y_\bullet} \to \abs{\Sigma^\infty_+ Y_\bullet} \simeq \un$ is built out of the objects $C_j[j]$, $j>i$, under colimits.
Each $C_j[j]$ is $j$-connective, as a direct summand of $\Sigma^\infty_+Y_i[j]$, and hence $(i+1)$-connective.
Hence $F_i[1]$ is $(i+1)$-connective, i.e. $F_i$ is $i$-connective.

\textit{Proof of \itemref{item:Fi Post}.}
More generally, every $\sF \in \sC$ built out of finite colimits and limits from $\Sigma^\infty_+ U$, where $U \to X$ is finitely presented \'etale, is Postnikov-complete.
This follows from \cite[Prop.~2.10]{ClausenMathew}, since $\sF$ is hypercomplete by assumption and, by \lemref{lem:coh dim heart}, the objects $\Sigma^\infty_+ U \in \sC$ are of cohomological dimension $\le c$ for every finitely presented \'etale $U \to X$.

\section{Lisse extensions and forgetting supports}\label{sec:fsupp}

We fix a presentable weave $\D$ on algebraic spaces and consider its lisse extension to Artin stacks.
For every $1$-Artin morphism $f : \sX \to \sY$ with proper diagonal, there exists a natural transformation $\alpha_f : f_! \to f_*$, called \emph{forgetting supports} (see \cite[\S 2.2]{weavelisse}).
In this section we provide sufficient conditions for this map to be invertible when $f$ is proper.
In particular, this will imply that $f$ admits $*$-direct image in $\D$, in the sense of \cite[\S 1.4]{weavelisse}.

\ssec{Descendability}

We begin with the following stacky generalization of \thmref{thm:descendability}.
Note that this does not require $\D$ to be lisse-extended from algebraic spaces.

\begin{thm}\label{thm:descendability DM}
  Let $\D$ be a presentable weave on Artin stacks satisfying localization, and suppose every finite surjection which is either étale or radicial is descendable.
  Let $\sY$ be a qcqs Deligne--Mumford stack and $f : \sX \twoheadrightarrow \sY$ a representable finitely presented surjection.
  Then $f$ is descendable.
\end{thm}
\begin{proof}
  Let $\sC \sub \mrm{Art}_{/\sY}$ be the full subcategory of Artin stacks which are étale and finitely presented over $\sY$, and denote by $\sC_0 \sub \sC$ the full subcategory spanned by $\sY' \to \sY$ such that $f_{\sY'} := f \fibprod_\sY \sY'$ is descendable.
  We will show that $\sC_0 = \sC$, so in particular $f$ is descendable, by applying the dévissage theorem of \cite[Thm.~D]{RydhDevissage}.

  Since $\sY$ is a quasi-compact DM stack, there exists an affine scheme $Y$ and a representable étale surjection $y : Y \twoheadrightarrow \sY$.
  Then $Y \in \sC$ and $f_Y$ is a finitely presented surjection of algebraic spaces, so $Y \in \sC_0$ by \thmref{thm:descendability}.

  It remains to verify the following closure conditions from \cite[Thm.~D]{RydhDevissage}:
  \begin{itemize}
    \item For any morphism $\sY'' \to \sY'$ in $\sC$ with $\sY' \in \sC_0$, we have $\sY'' \in \sC_0$.
    Indeed, $f_{\sY'}$ is descendable, hence so is its base change $f_{\sY'} \fibprod_{\sY'} \sY'' = f_{\sY''}$.
    \item For any finite étale surjection $g : \sY'' \to \sY'$ in $\sC$ with $\sY'' \in \sC_0$, we have $\sY' \in \sC_0$.
    Indeed, $g$ is finite étale, so it satisfies the projection formula and base change, and it is descendable by assumption; since $f_{\sY''}$ is the base change of $f_{\sY'}$ along $g$, \corref{cor:descendable local} shows that $f_{\sY'}$ is descendable.
    \item For $\sY' \in \sC$, any open immersion $\sU \hook \sY'$ with $\sU \in \sC_0$, and any étale neighbourhood $\sV \to \sY'$ of $\sY' \setminus \sU$ with $\sV \in \sC_0$, we have $\sY' \in \sC_0$.
    Indeed, the étale neighbourhood lifts $\sZ = (\sY' \setminus \sU)_{\red}$ to a closed substack of $\sV$, so $f_{\sV} \fibprod_{\sV} \sZ = f_\sZ$ is descendable by base change.
    By \lemref{lem:descendability and localization}, the descendability of $f_{\sU}$ and $f_{\sZ}$ then yields descendability of $f_{\sY'}$.\qedhere
  \end{itemize}
\end{proof}

\ssec{Forgetting supports, DM case}

\begin{thm}\label{thm:forget supp wild DM}
  Let $\D$ be a lisse-extended presentable weave on Artin stacks satisfying localization, and suppose every finite surjection which is either étale or radicial is descendable.
  If $f : \sX \to \sY$ is a proper Deligne--Mumford morphism of $1$-Artin stacks, then $\alpha_f : f_! \to f_*$ is invertible.
\end{thm}
\begin{proof}
  Let $(Y_i \to \sY)_\alpha$ be a smooth covering family where $Y_i$ are affine schemes, and consider the base changes $f_i : \sX \fibprod_\sY Y_i \to Y_i$.
  Since the $*$-pull-backs to $Y_i$ are jointly conservative and send $(\alpha_f$ to $\alpha_{f_i}$ by smooth base change, we may assume $\sY$ is an affine scheme.
  Then $\sX$ is a quasi-compact DM stack with quasi-compact and separated diagonal.

  Now by \cite[Thm.~B]{RydhApprox} there exists a finitely presented finite surjection $g : Z \twoheadrightarrow \sX$ with $Z$ a scheme.
  By \thmref{thm:descendability DM}, $g$ is descendable.
  Denote by $Z_\bullet$ the \v{C}ech nerve of $g$; its terms are algebraic spaces since $g$ is representable.
  By proper descent along $g$ (\propref{prop:proper descendable implies descent}), every $\sF \in \D(\sX)$ is the totalization of $g_{\bullet,*}g_\bullet^*(\sF)$.
  This totalization is preserved by the limit-preserving functor $f_*$ and, since $g$ is descendable, also by the exact functor $f_!$ (\remref{rem:weave totalization commutes}).
  Since each $g_\bullet : Z_\bullet \to \sX$ is proper representable and each $f \circ g_\bullet : Z_\bullet \to Y$ is a proper morphism of algebraic spaces, $\alpha_{g_\bullet}$ and $\alpha_{f\circ g_\bullet}$ are invertible by \cite[Lem.~2.21]{weaves} and \cite[Thm.~3.4.5(ii), Lem.~2.2.2]{weavelisse}.
  The claim follows by functoriality of $\alpha$.
\end{proof}

\begin{exam}
  For example, for $\SH_{\Q,+}$ all finite étale surjections and all finite radicial surjections are descendable of index $\le 1$.
  Indeed, for finite étale surjections the Euler characteristic \eqref{eq:fetcons} is invertible, and for finite radicial surjections $f$, the unit $\id \to f_*f^*$ is invertible (both statements are checked by descent from the case of schemes).
\end{exam}

\ssec{Forgetting supports, tame Artin case}

Suppose next that $\sY$ is not necessarily DM but only \emph{$1$-Artin}.
We are then able to prove a result analogous to \thmref{thm:forget supp wild DM} in the ``tame'' case.
For this we only require the following slightly weaker conditions on $\D$:
\begin{enumerate}
  \item[(D1)] Let $f : \sY \twoheadrightarrow \sX$ be a torsor under a finite étale group scheme $H$ over $\sX$ whose order $|H|$ is invertible in $\End_{\D(\sX)}(\un_\sX)$.
  Then $f$ is descendable.
  \item[(D2)] Let $f : \sY \twoheadrightarrow \sX$ be a finite\footnote{
    We recall that a morphism of stacks $f : \sX \to \sY$ is \emph{finite} if it is schematic, and the base change $f \fibprod_\sY Y$ is a finite morphism of schemes for every scheme $Y$ over $\sY$.
  } radicial surjection such that every prime number is invertible either in $\sO_\sX$ or in $\End_{\D(\sX)}(\un_\sX)$.
  Then $f$ is descendable.
\end{enumerate}

\begin{thm}\label{thm:forget supp tame Artin}
  Let $\D$ be a lisse-extended presentable weave on Artin stacks satisfying localization and (D1) and (D2).
  Let $f : \sX \to \sY$ be a proper $1$-Artin morphism with proper diagonal.
  If $f$ has tame relative inertia\footnote{
    That is: the relative inertia is finite, and for every geometric point $y$ of $\sY$, the fibre $\sX_y$ has linearly reductive stabilizers (and hence is a tame Artin stack in the sense of \cite{AOVtame}).
  }, stabilizer orders invertible in $\End_{\D(\sY)}(\un_\sY)$, and $\sY$ equicharacteristic, then $\alpha_f : f_! \to f_*$ is invertible.
\end{thm}

\begin{proof}
  As in the proof of \thmref{thm:forget supp wild DM} we may assume $\sY=Y$ is an equicharacteristic algebraic space and $\sX$ is $1$-Artin with finite inertia.
  Then by Keel--Mori there exists a coarse moduli space $\pi : \sX \to M$, and $f : \sX \to Y$ factors as the proper surjection $\pi$ followed by a proper morphism of algebraic spaces $g : M \to Y$.
  Since $\alpha_g$ is an isomorphism by \cite[Lem.~2.21]{weaves} and $\alpha$ is functorial in $f$, we may replace $f$ by $\pi$ and assume that $f : \sX \to M$ is the (equicharacteristic) coarse moduli space of a tame $1$-Artin stack $\sX$.

  Since $\sX$ is tame, there exists by \cite[Thm.~3.2]{AOVtame} an étale surjection of algebraic spaces $p : M' \twoheadrightarrow M$ such that $\sX \fibprod_M M' \simeq [U/G]$ where $U$ is a scheme of finite presentation over $M'$ and $G$ is a linearly reductive finite fppf group scheme over $M'$ acting on $U$.
  Moreover, the proof of \emph{loc. cit.} shows that over each connected component of $M'$ there is a geometric point $x$ of $\sX$ such that the fibre $G_{f(x)}$ is the stabilizer $\underline{\Aut}_\sX(x)$.
  Since $G \to M'$ is finite locally free, its order $|G|$ is locally constant and hence is equal, on the component containing $x$, to the order of this stabilizer; in particular it is invertible in $\End_{\D(M')}(\un_{M'})$ by hypothesis.

  By smooth base change again we may replace $M$ by $M'$ and thereby assume that $\sX = [U/G]$, with $G$ finite linearly reductive of order invertible in $\End_{\D(M)}(\un_M)$.
  By \cite[Thm.~19.9\,(2,6)]{AHREtaleLocal} (which apply since $M$ is still equicharacteristic), there is a short exact sequence $1 \to G^0 \to G \twoheadrightarrow H \to 1$ of group schemes over $M$ with $H$ tame finite étale and $G^0$ finite radicial.
  Thus the quotient map $q : U \twoheadrightarrow [U/G]$ factors through the finite radicial surjection $U \twoheadrightarrow [U/G^0]$ followed by the finite étale surjection $[U/G^0] \to [U/G]$ (which is an $H$-torsor).
  Since $|H|$ and $|G^0|$ divide $|G|$, they are also invertible in $\End_{\D(M)}(\un_M)$.
  Thus assumptions (D1) and (D2) imply that $q$ is descendable.
  For $U \twoheadrightarrow [U/G^0]$, note that (D2) applies because every prime is invertible in $\sO_M$ or in $\End_{\D(M)}(\un_M)$.
  Indeed, since $M$ is equicharacteristic, every prime other than its residue characteristic is invertible in $\sO_M$; in characteristic $p > 0$ the finite radicial group $G^0$ has $p$-power order, so either it is trivial, or $p$ divides $|G^0|$ and is invertible in $\End_{\D(M)}(\un_M)$.

	Denote by $U_\bullet$ the \v{C}ech nerve of $q : U \twoheadrightarrow [U/G] = \sX$.
	By proper descent along $q$ (\propref{prop:proper descendable implies descent}), every $\sF \in \D(\sX)$ is the totalization of the cosimplicial diagram $q_{\bullet,*} q_{\bullet}^*(\sF)$.
	We now conclude just as in the proof of \thmref{thm:forget supp wild DM}, using the descendability of $q$ proven above.
\end{proof}

Institute of Mathematics, Academia Sinica, 10617 Taipei, Taiwan

\end{document}